\documentclass[a4paper,9pt]{amsart}
\usepackage{graphicx,amssymb,amstext,amsmath,latexsym,dsfont,comment,color}
\newtheorem{theorem}{Theorem}
\newtheorem{lemma}[theorem]{Lemma}

\newcommand{\R}{\mathbb{R}}
\newcommand{\N}{\mathbb{N}}

\begin{document}

\title[On the heat content of a polygon]{On the heat content of a polygon}
\date{29 July 2016}

\author[M. van den Berg, K. Gittins]{M. van den Berg\ *, K. Gittins}
 \thanks{*\ Partially supported by The Leverhulme Trust,
International Network Grant \emph{Laplacians, Random Walks, Bose
Gas, Quantum Spin Systems}}
\address{School of Mathematics\\ University of Bristol\\
University Walk\\ Bristol\\ BS8 1TW, U.K.}
\keywords{Heat content, polygon, fractal polyhedron} \subjclass[2010]{35A99}
\subjclass[2010]{Primary: 35K05; Secondary: 35K20}
\begin{abstract}
Let $D$ be a bounded, connected, open set in Euclidean space $\R^{2}$ with polygonal
boundary. Suppose $D$ has initial temperature $1$ and the complement of $D$ has initial
temperature $0$. We obtain the asymptotic behaviour of the heat content of $D$ as time
$t \downarrow 0$. We then apply this result to compute the heat content of a particular
fractal polyhedron as $t \downarrow 0$.
\end{abstract}
\maketitle

\section{Introduction}\label{S1}
The conduction of heat or the diffusion of matter through a solid body is of importance in the physical and engineering sciences.
The classic reference of Carslaw and Jaeger, \cite{Ca45}, analyses many examples and applications. The mathematical tools used in
\cite{Ca45} are centred around separation of variables and Laplace transforms and, in many cases, require properties of special functions.
From a mathematical point of view, the heat equation, heat content and heat trace link the underlying geometry of the manifold and its
boundary and boundary conditions to the spectral resolution of the Laplace operator. Over the last few decades, a considerable amount of
progress has been made in understanding the asymptotic behaviour of the heat content for small time $t$, see \cite{G}.

It was discovered by Preunkert, \cite{P}, that even in the absence
of boundary conditions the heat content of a ball $B$ in Euclidean
space $\R^m$  which is at initial temperature $1$, while $\R^{m} - B$
has initial temperature $0$, has non-trivial asymptotic behaviour as
$t\downarrow 0$. For small $t$, the initial condition on the complement
of $B$ acts in a similar way to a Dirichlet $0$ boundary condition.
This was subsequently stated for bounded, open sets with $C^{1,1}$ boundary
in \cite{MPPP1}, and proved in \cite{MPPP07}. The
discussion is simplified by the fact that the heat kernel on
$\R^m$ is known explicitly. The general situation for the heat
content of a compact subdomain $\Omega$ in a compact Riemannian
manifold $M$ was examined in \cite{vdBG2}. The tools of
pseudo-differential calculus used there rely heavily upon the smoothness
assumptions on the boundary. Two-sided estimates for the heat
content of non-compact sets in $\R^m$ were obtained in
\cite{mvdB13}. These estimates are very different from the ones
where Dirichlet $0$ boundary conditions are imposed. See
\cite{mvdB2006} and \cite{mvdB2007}.

In this paper we denote the
fundamental solution of the heat equation on $\R^{m}$ by
\begin{equation*}
p(x,y;t)=(4\pi t)^{-m/2}e^{-\vert x - y \vert^{2}/(4t)},
\end{equation*}
and for an open set $D \subset \R^{m}$, we define
\begin{equation}
u_{D}(x;t)=\int_{D} dy \, p(x,y;t). \label{e2}
\end{equation}
Then $u_{D}(x;t)$ satisfies the heat equation on $\R^{m}$
\begin{equation}\label{e3}
\Delta u_{D} = \frac{\partial u_{D}}{\partial t}, \quad x \in
\R^{m}, \, t>0,
\end{equation}
(see Chapter 2 in \cite{LCE}) and
\begin{equation}\label{e4}
\lim_{t \downarrow 0}u_D(x;t)= \mathds{1}_{D}(x), \quad x \in
\R^{m}- \partial D,
\end{equation}
where $\partial D$ is the boundary of $D$.

We define the \emph{heat content of $D$ in $\R^{m}$ at
$t$} by
\begin{equation*}
H_{D}(t)=\int_{D} dx \, u_{D}(x;t).
\end{equation*}
So by \eqref{e2},
\begin{equation}\label{e5a}
H_D(t)=\int_{D} dx \int_{D}\ dy \, p(x,y;t).
\end{equation}

We denote the Lebesgue measure of a measurable set $A\subset \R^m$ by $|A|$,
its perimeter by $\mathcal{P} (A)$, and its $(m-1)$-dimensional Hausdorff measure by
$\mathcal {H}^{m-1}(A)$.

If $D$ is a bounded, open set in $\R^{m}$, $m \geq 2$, with
$C^{1,1}$ boundary $\partial D$, then Theorem 2.4 of \cite{MPPP07}
implies that
\begin{equation}\label{e6}
H_{D}(t)=\vert D \vert - \mathcal{P}(D)\frac{t^{1/2}}{\sqrt{\pi}} +
o(t^{1/2}),\ t\downarrow 0.
\end{equation} In
\cite{vdBkG14}, we obtained explicit bounds for $H_{D}(t)$ for
bounded, open sets $D$ in Euclidean space with $C^{1,1}$ boundary
which are uniform in $t$ and in the geometric data of $D$. These
bounds imply that
\begin{equation}\label{e6a}
H_{D}(t)=\vert D \vert - \mathcal{P}(D)\frac{t^{1/2}}{\sqrt{\pi}} +
O(t),\ t\downarrow 0.
\end{equation}
We observe that if $K$ is a closed set in $\R^m$ with $|K|=0$ then by the definition of the perimeter (see \cite{EG92}) and \eqref{e5a},
\begin{equation}\label{e5b}
\vert D-K\vert =\vert D\vert,\ \ \mathcal{P}(D-K)=\mathcal{P}(D),\ \   H_{D- K}(t)=H_D(t).
\end{equation}
The observations in \eqref{e5b} suggest that \eqref{e6} holds for all open sets $D$ with
finite perimeter and finite Lebesgue measure. The proof of such a statement is well beyond the scope of this paper.

In this paper, we focus on the heat content of a bounded, connected, open set $D \subset \R^{2}$
with polygonal boundary. We introduce some notation and then present the main result: Theorem~\ref{T1}.
Let $\gamma_1,\cdots, \gamma_n$ denote the interior angles of $\partial D$. Each such angle $\gamma_{j}$
is supported by two edges provided $\gamma_{j}<2\pi.$ By \eqref{e5b}, we may exclude angles
$2\pi.$ We label the corresponding vertices by $V_1,\cdots, V_n$ and note that these $n$ vertices need not
be pairwise disjoint. Let $W_{j}$ denote the infinite wedge of angle $\gamma_{j}$ with vertex $V_{j}$ such
that $W_{j} \cap D \neq \emptyset$ and the boundary of the wedge contains the two edges which are adjacent
to $V_{j}$ and have an angle $\gamma_{j}$. Let
\begin{equation}\label{e7}
\gamma = \{\gamma_{i} \, : \, (\sin \gamma_{i})^{2} \leq (\sin \gamma_{j})^{2} \, \text{ for all } j \in \{1,2,\cdots,n\}\}.
\end{equation}
For $r > 0$, we also define the open sector
\begin{equation}\label{e8}
B_{j}(r) = \{ x \in W_{j} \, : \, d(x,V_{j}) < r\}
\end{equation}
and
\begin{equation}\label{e9}
R=\frac{1}{2} \sup \left\{r \, : \, B_{\ell}(r) \cap B_{j}(r) = \emptyset \text{ for all } \ell \neq j, \,
\bigcup_{k=1}^{n} B_{k}(r) \subset D\right\}.
\end{equation}
\begin{theorem}\label{T1}
Let $D \subset \R^{2}$ be a bounded, connected, open set with polygonal boundary
$\partial D$ with $\gamma$ and $R$ as defined in \eqref{e7} to \eqref{e9}.
Then as $t \downarrow 0$,
\begin{align}
H_{D}(t) &= \vert D \vert - \mathcal{P}(D)\frac{t^{1/2}}{\sqrt{\pi}} +
\sum_{j=1}^{n} g(\gamma_j)t\label{e10}\\
&\ \ \ +\sum_{j,\ell \in \{1, \dots, n : j \neq \ell, V_{j}=V_{\ell}\}}
k(\alpha_{j},\gamma_{j},\gamma_{\ell})t + O(e^{-R^{2}(\sin \gamma)^{2}/(32t)}),\notag
\end{align}
where $g: (0,2\pi) \rightarrow \R$ is given by
\begin{equation}\label{e11a}
g(\beta) =\begin{cases} \frac{1}{\pi} + \left(1-\frac{\beta}{\pi}\right)
\cot \beta, &\beta \in (0,\pi)\cup(\pi,2\pi);\\
0, &\beta=\pi, \end{cases}
\end{equation}
and
$k: (0,\pi) \times (0,2\pi) \times (0,2\pi) \rightarrow \R$ is given by
\begin{align}
k(\alpha, \theta, \sigma)
&=\frac{1}{2\pi}\left(-(\sigma+\theta+\alpha-\pi)\cot(\sigma+\theta+\alpha)-(\alpha-\pi)\cot\alpha\right)\notag\\
&+\frac{1}{2\pi}\left((\sigma+\alpha-\pi)\cot(\sigma+\alpha)+(\theta+\alpha-\pi)\cot(\theta+\alpha)\right),\
\label{e11b}
\end{align}
for $\sigma+\theta+\alpha\ne \pi,\ \alpha\ne \pi,\ \sigma+\alpha\ne \pi,\ \theta+\alpha\ne \pi$,
where $\alpha$ denotes the smallest angle between $W_{\theta}$ and $W_{\sigma}$. In any of the remaining cases,
such as $\alpha=\pi$, we define $k(\alpha, \theta, \sigma)$ by taking appropriate limits using l'H\^{o}pital's rule.
\end{theorem}

The terms which involve the area and perimeter are as expected and agree with those in \eqref{e6a}. We see
that the heat content has a non-trivial dependence on the interior angles
of the polygonal boundary.

We observe that $\beta\mapsto g(\beta)$ is continuous on $(0,2\pi)$, decreasing on $(0,\pi]$
and symmetric with respect to $\pi$. That is
\begin{equation}\label{e12a}
g(\beta)=g(2\pi-\beta),\ 0<\beta<2\pi.
\end{equation}
By \eqref{e11a} and \eqref{e12a}, we conclude that $g$ is non-negative.
We remark that $k(\alpha, \theta, \sigma)$ is symmetric with respect to
$\theta$ and $\sigma$ and that $k(\alpha, \theta, \sigma)=
k(2\pi-\theta-\sigma-\alpha, \theta, \sigma)$. By Lemma \ref{L3.2},
Section~\ref{S3}, it follows that $k$ is non-negative.

In addition, if $D$ is a regular $n$-gon in $\R^{2}$, then
$\gamma_1=\gamma_2=\cdots =\gamma_n = \left(\frac{n-2}{n}\right)\pi$, and
the angular contribution to the heat content is
\begin{equation*}
\frac{n}{\pi} + 2\cot\left(\left(\frac{n-2}{n}\right)\pi\right)=O\left(\frac{1}{n}\right), \, n \rightarrow \infty.
\end{equation*}
We observe that the coefficient of $\int_{\partial
\Omega} L_{aa}t$ in the expansion of the heat content of a compact
domain $\Omega$ with smooth boundary $\partial \Omega$ is also equal to
$0$, see Theorem 1.6 in \cite{vdBG2}. Here $L_{aa}$ is the trace
of the second fundamental form when $\partial {\Omega}$ is
oriented with a smooth inward-pointing unit normal vector field.

The expansion for the heat content of a polygon with Dirichlet $0$
boundary conditions was obtained in \cite{vdBS90}. There it was shown that
if $v_D$ solves the heat equation $\Delta v = \frac{\partial v}{\partial t}, \quad x \in
D, \, t>0$ with $\lim_{t \downarrow 0}v(x;t)=1, \quad x \in D$ and satisfies a Dirichlet boundary
condition $\lim_{x\rightarrow x_0}v(x;t)=0$ for any $x_0\in \partial D$ then
\begin{equation}\label{e13a}
\int_D dx \, v_D(x;t)= |D|- 2\mathcal{P}(D)\frac{t^{1/2}}{\sqrt{\pi}} +
\sum_{j=1}^{n} c(\gamma_j)t+ O(e^{-R^{2}(\sin (\gamma/2))^{2}/(32t)}),
\end{equation}
where
\begin{equation*}
c(\beta)=\int_0^{\infty}\frac{4\sinh((\pi-\beta)x)}{(\sinh(\pi x))(\cosh(\beta x))}dx.
\end{equation*}
We note that both \eqref{e10} and \eqref{e13a} have angular contributions which are additive.
However, in Theorem \ref{T1} there is an additional term in the case where vertices
have multiplicity larger than $1$. No such term is present in \eqref{e13a} since sectors based at
the same vertex do not feel each other's presence due to the Dirichlet $0$ boundary condition.

The strategy to prove \eqref{e13a} is inspired by \cite{K}, and
relies on some model computations. We use an analogous strategy
to prove \eqref{e10}. For points $x\in D$ close to a
vertex, say $x \in B_{j}(r)$ for some $j \in \{1,2,\cdots,n\}$, $r>0$,
$u_D$ is approximated by $u_{W_{j}}$. For points $x\in D$ which have a
distance at least $\delta$ to $\partial D$, for some $\delta>0$,
$u_D$ is approximated by $1$. For the remaining points in $D$,
$u_D$ is approximated by $u_{H}$, where $H$ is the half-plane
which contains $D$ and whose boundary  contains the
edge of $\partial D$ nearest to $x$.  As was the case in \cite{vdBS90},
the model computations involving the infinite wedge $W_{j}$ are the most
difficult to carry out. However, in contrast to the case with Dirichlet 0
boundary conditions, we must also consider the contribution to the heat content
from a vertex which belongs to the boundaries of more than one wedge.
We deal with these computations in Lemma \ref{L3.1} and Lemma \ref{L3.2}
in Section \ref{S3}. In Section \ref{S4} we carry out the half-plane computations.

It is quite remarkable that, in contrast to the smooth case \eqref{e6},
the asymptotic expansion in half powers of $t$ of $H_D(t)$ in Theorem \ref{T1} terminates after the term of order $t$, leaving an exponentially
small remainder as $t \downarrow 0$. This agrees with the fact that there are no further locally computable invariants of $D$ and $\partial D$ available from which
non-trivial quantities could be built. A similar phenomenon has been observed for the asymptotic expansion of the heat trace. See for example \cite{vdBS88}.
The precise form of the exponential remainder remains an open problem.
The extension of the results in this paper to general polyhedra
in $\R^3$ is another challenge beyond the scope of this paper.

However, in Section~\ref{S7}, we use Theorem~\ref{T1} to compute the heat content of a fractal polyhedron
which is constructed as follows. Let $Q_{0} \subset \R^{3}$ be an open cube of
side-length 1. Let $0<s<1$. Attach a regular open cube $Q_{1,i}$ of side-length $s$ to the
centre $c_{1,i}, i=1,\dots,6,$ of each face of $\partial Q_{0}$, and such that all the faces are pairwise-parallel.
Now proceed by induction. For $j=2,3,\dots,$ attach $N(j)=6\cdot5^{j-1}$ open cubes $Q_{j,1}, \dots, Q_{j,N(j)},$
of side-length $s^{j}$ to the centres of the boundary faces of the cubes $Q_{j-1,1},\dots,Q_{j-1,N(j-1)}$,
again with pairwise-parallel faces. We define the fractal polyhedron $D_{s}$ as
\begin{equation*}
D_{s}=\text{interior}\left\{\overline{ Q_{0} \cup \left[\bigcup_{j\geq1} \bigcup_{1\leq i \leq N(j)} Q_{j,i} \right]}\right\}
\end{equation*}
for $0<s<\sqrt{2}-1$ (see Figure~\ref{fig1}). We note that for this range of $s$, no cubes in the construction of $D_{s}$ overlap.
For the two-dimensional construction, see Theorem 4 in \cite{vdBdH99}. In that paper, the critical value $s=\sqrt{2}-1$,
where the squares just touch, is allowed. This is due to the fact that the Dirichlet $0$ boundary conditions guarantee the independence
of the heat flow in these touching squares. In this paper, we do not impose Dirichlet $0$ boundary conditions on $\partial D_{s},$ and if
the cubes touch, then this could give rise to an extra term. For this reason, we only allow $0<s<\sqrt{2}-1$.

\begin{figure}[h]
\centering\includegraphics[scale=.9]{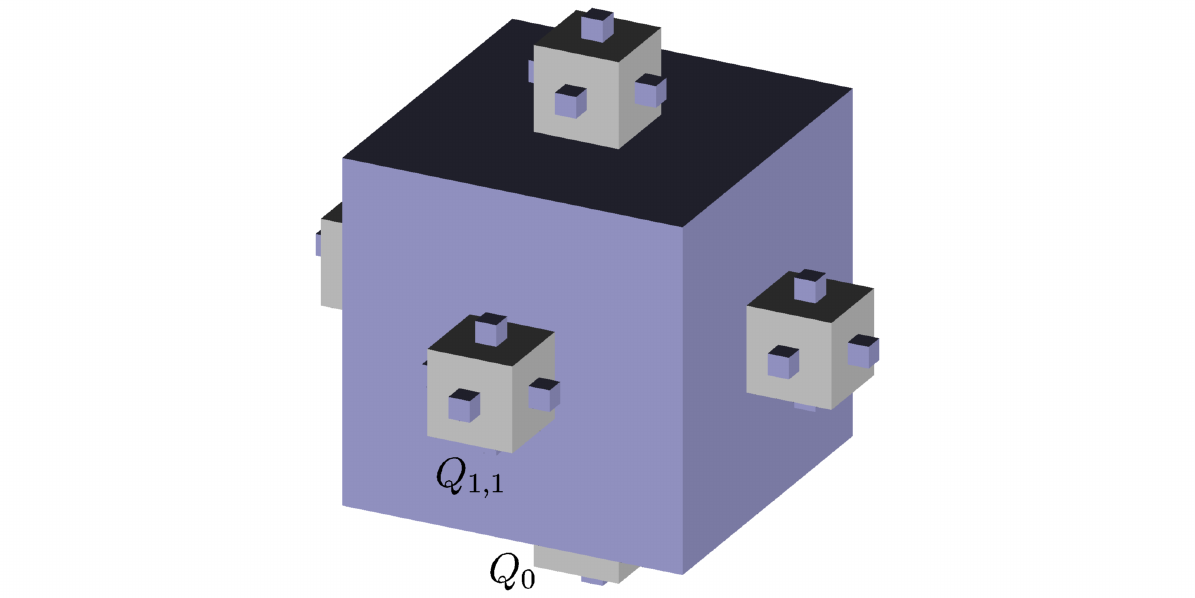}
\caption{The first two generations of $D_{s}$ with $s=\frac{1}{4}$.}\label{fig1}
\end{figure}

We have that
\begin{equation*}
\vert D_{s} \vert = \frac{1+s^{3}}{1-5s^{3}}, \
\end{equation*}
and that the two-dimensional Hausdorff measure of the boundary is given by
\begin{equation*}
\mathcal{H}^2(\partial D_{s})  = 6\left(\frac{1-s^{2}}{1-5s^{2}}\right).
\end{equation*}
Both quantities are finite for $0<s<\sqrt{2}-1$. Moreover, the total length of the edges of $\partial D_{s}$ is finite if and only if $0<s<\frac15$ and equals
$12\left(\frac{1+s}{1-5s}\right)$.
In addition to Theorem~\ref{T1}, in order to compute the heat content of $D_{s}$,
we require the heat content of a sector of angle $\pi$ in an infinite wedge of angle $\frac{3\pi}{2}$
where this sector and wedge share one common edge and vertex. This will be computed in Section~\ref{S6}.
In Section~\ref{S7}, we prove the following theorems which give the asymptotic expansion for
the heat content of $D_{s}$ as $t \downarrow 0$.
\begin{theorem}\label{T2}
Let $d=\frac{3}{2} + \frac{1}{2}\frac{\log 5}{\log s}$. Fix $0<s<\sqrt{2}-1$, $s \neq \frac{1}{5}$.
There exists a periodic, continuous function $p_{s}: \R \rightarrow \R$ with period $\log (s^{-2})$ such that
\begin{align}
H_{D_{s}}(t) &= \frac{1+s^{3}}{1-5s^{3}} - 6\left(\frac{1-s^{2}}{1-5s^{2}}\right)\frac{t^{1/2}}{\sqrt{\pi}}
+\frac{12}{\pi}\left(\frac{1+s}{1-5s}\right)t \label{e13f}\\
&\ \ \ + 6p_{s}(\log t)t^{d} + O(t^{3/2}\left(\log(t^{-1})\right)^{3/2}),\ t \downarrow 0.\notag
\end{align}
\end{theorem}

It is easy to see that if we write $d=\frac{m-d_s}{2},\ m=3$, in Theorem \ref{T2} then
$d_s=\frac{\log 5}{\log (s^{-1})}$ is the interior Minkowski dimension of the vertices of $\partial D_{s}$
for $0<s<\frac15$, whereas for $\frac15<s<\sqrt 2-1$, it is the interior Minkowski dimension of the edges of
$\partial D_{s}$. Below we state the corresponding result for the critical case $s=\frac15$.

\begin{theorem}\label{T3}
For $s=\frac{1}{5}$, there exists a periodic, continuous function $p_{\frac{1}{5}}: \R \rightarrow \R$ with period $\log 25$
such that,
\begin{align}
H_{D_{\frac{1}{5}}}(t)&= \frac{21}{20} - \frac{36}{5}\frac{t^{1/2}}{\sqrt{\pi}} + \frac{132}{5\pi}t
-\frac{36}{5\pi\log 5} t \log t \label{e13g}\\
&\ \ \ + 6p_{\frac{1}{5}}(\log t)t + O(t^{3/2}\left(\log(t^{-1})\right)^{3/2}),\ t \downarrow 0.\notag
\end{align}
\end{theorem}
In Section \ref{S2} below we introduce some further notation and
state and prove several lemmas.

\section{Additional notation and lemmas.}\label{S2}
Let $D$ be as given in Theorem~\ref{T1}.
In the proofs of Lemma \ref{L1} to Lemma \ref{L4} we use a variant of Kac's principle of not feeling the
boundary to reduce the computation of the heat content of $D$ to a collection of model computations.
For $r>0$, $\delta >0$, $x \in D$, $A \subset \R^{2}$ we define
\begin{equation}\label{e14a}
d(x,A) = \inf \{\vert x - z \vert : z \in A\},
\end{equation}
and
\begin{equation*}
C(\delta,r) = \left\{x \in D \, : \, d(x, \partial D)<\delta, \, x \not\in \bigcup_{k=1}^{n} B_{k}(r) \right\},
\end{equation*}
and
\begin{equation*}
D(\delta,r) = \left\{x \in D \, : \, x \not\in \bigcup_{k=1}^{n} B_{k}(r), \, x \not\in C(\delta,r)\right\}.
\end{equation*}
Choose $R$ as in \eqref{e9} and let $\delta = \frac{R}{2} \vert \sin \gamma \vert$.
\begin{lemma}\label{L1}
\begin{equation*}
\int_{D(\delta,R)}dx \int_{D} dy \, p(x,y;t) = \vert D(\delta, R) \vert + O(e^{-R^{2}(\sin \gamma)^{2}/(32t)}), \,
t \downarrow 0.
\end{equation*}
\end{lemma}

\begin{proof}
By \cite[Proposition 9(i)]{mvdB13}, we have that
\begin{equation*}
1-2e^{-\delta(x)^{2}/(8t)} \leq \int_{D} dy \, p(x,y;t) \leq 1,
\end{equation*}
where $\delta(x) = \min \{ \vert x-y \vert \, : \, y \in \R^{2}-D\}$.
Since $x \in D(\delta, R)$, $\delta(x) \geq \delta$.
Hence
\begin{equation*}
\vert D(\delta,R)\vert - 2 \vert D(\delta,R)\vert e^{-\delta^{2}/(8t)}
\leq \int_{D(\delta,R)}dx \int_{D} dy \, p(x,y;t) \leq \vert D(\delta, R)\vert
\end{equation*}
as required.
\end{proof}

We partition the region $D - D(\frac{R}{2} \vert \sin \gamma \vert, R)$ into $n$ sectors, $B_{i}(R)$,
of radius $R$, $n$ rectangles, $S_{\gamma}$, of height $\frac{R}{2} \vert \sin \gamma \vert$
and $2n$ cusps of height $\frac{R}{2} \vert \sin \gamma \vert$. The contributions to $H_{D}(t)$
from these regions will be computed in Sections~\ref{S3}, \ref{SS4} and \ref{SS5} respectively.
Each sector has two neighbouring cusps. Each cusp is adjacent to a rectangle and a sector.
The corresponding $m$-dimensional result to \cite[Lemma 7]{vdBdH99} is the following.

\begin{lemma}\label{L}
Let $\tilde{D},F,G$ be non-empty, open subsets of $\R^{m}, m \geq 2$ such that $\tilde{D} \cap F \neq \emptyset$ and
$G \subset \tilde{D} \cap F$. Let $E$ be a bounded, measurable subset of $G$. Then
\begin{equation*}
\int_{E}dx \int_{\tilde{D}} dy \, p(x,y;t) = \int_{E}dx \int_{F} dy \, p(x,y;t) + O(e^{-\epsilon^{2}/(8t)}),
\, t \downarrow 0,
\end{equation*}
where
\begin{equation*}
\epsilon=\inf \{\vert x-y \vert : x \in E, y \in \overline{(\tilde{D} \cup F) \cap \partial G}\}.
\end{equation*}
\end{lemma}

\begin{proof}
We write
\begin{align}
&\int_{E}dx \int_{\tilde{D}} dy \, p(x,y;t) \notag\\
&=\int_{E}dx \int_{\tilde{D} \cap F} dy \, p(x,y;t) + \int_{E}dx \int_{\tilde{D} \cap F^{c}} dy \, p(x,y;t) \notag\\
&=\int_{E}dx \int_{F} dy \, p(x,y;t)-\int_{E}dx \int_{\tilde{D}^{c} \cap F} dy \, p(x,y;t)
+ \int_{E}dx \int_{\tilde{D} \cap F^{c}} dy \, p(x,y;t).\label{e18c}
\end{align}
By \eqref{e18c}, we have that
\begin{align*}
\int_{E}dx \int_{\tilde{D}} dy \, p(x,y;t)&\geq \int_{E}dx \int_{F} dy \, p(x,y;t)-\int_{E}dx \int_{\tilde{D}^{c} \cap F} dy \, p(x,y;t) \notag\\
&\geq \int_{E}dx \int_{F} dy \, p(x,y;t) -2^{m/2}\vert E \vert e^{-\epsilon^{2}/(8t)}, 
\end{align*}
and
\begin{align*}
\int_{E}dx \int_{\tilde{D}} dy \, p(x,y;t)&\leq \int_{E}dx \int_{F} dy \, p(x,y;t)+\int_{E}dx \int_{\tilde{D} \cap F^{c}} dy \, p(x,y;t) \notag\\
&\leq \int_{E}dx \int_{F} dy \, p(x,y;t) +2^{m/2}\vert E \vert e^{-\epsilon^{2}/(8t)}. 
\end{align*}
This completes the proof.
\end{proof}

\begin{lemma}\label{L2}
Let $i \in \{1,\cdots,n\}$ and $k \in \N$ such that $i+k\leq n$.
Suppose $\gamma_{i}, \gamma_{i+1}, \cdots, \gamma_{i+k}$ are interior angles of $\partial D$
which are supported by edges which meet at the same vertex $V_{i}=V_{i+1}=\cdots=V_{i+k}$. Then
\begin{align}
&\int_{\cup_{j=i}^{i+k} B_{j}(R)}dx \int_{D} dy \, p(x,y;t) \label{e19}\\
&= \sum_{j=i}^{i+k}\int_{B_{j}(R)}dx \int_{W_{j}} dy \, p(x,y;t) +\sum_{j \neq \ell, j,\ell=i}^{i+k} \int_{W_{j}}dx \int_{W_{\ell}}dy \, p(x,y;t)
+ O(e^{-R^{2}(\sin \gamma)^{2}/(32t)}), \, t \downarrow 0.\notag
\end{align}
\end{lemma}

\begin{proof}
By Lemma~\ref{L} with $\tilde{D}=D$, $F=\cup_{j=i}^{i+k} W_{j}$, $E=\cup_{j=i}^{i+k} B_{j}(R)$ and
$G=\{z \in D : d(z, \cup_{j=i}^{i+k} B_{j}(R))<\frac{R}{2}\vert \sin \gamma \vert\},$
we have that
\begin{align}
&\int_{\cup_{j=i}^{i+k} B_{j}(R)}dx \int_{D} dy \, p(x,y;t) \label{e20}\\
&=\int_{\cup_{j=i}^{i+k} B_{j}(R)}dx \int_{\cup_{j=i}^{i+k} W_{j}} dy \, p(x,y;t)
+O(e^{-R^{2}(\sin \gamma)^{2}/(32t)}),\, t \downarrow 0.\notag
\end{align}
We also have that
\begin{align}
&\int_{\cup_{j=i}^{i+k} B_{j}(R)}dx \int_{\cup_{j=i}^{i+k} W_{j}} dy \, p(x,y;t) \label{e21}\\
&=\sum_{j=i}^{i+k}\int_{B_{j}(R)}dx \int_{W_{j}} dy \, p(x,y;t)
+\sum_{j \neq \ell, j,\ell=i}^{i+k} \int_{W_{j}}dx \int_{W_{\ell}}dy \, p(x,y;t) \notag\\
&+ O(e^{-R^{2}/(8t)}), \, t \downarrow 0.\notag
\end{align}
Combining \eqref{e20} and \eqref{e21} gives \eqref{e19}.
\end{proof}

\begin{lemma}\label{L3}
Let $H \subset \R^{2}$ denote the half-plane such that $\emptyset \neq \partial S_{\gamma} \cap \partial D \subset \partial H$
and $S_{\gamma} \subset H$. Then
\begin{equation*}
\int_{S_{\gamma}}dx \int_{D} dy \, p(x,y;t) = \int_{S_{\gamma}}dx \int_{H} dy \, p(x,y;t)
+O(e^{-R^{2}(\sin \gamma)^{2}/(32t)}), \, t \downarrow 0.
\end{equation*}
\end{lemma}

\begin{proof}
Using Lemma~\ref{L} with $\tilde{D}=D$, $F=H$, $E=S_{\gamma}$ and $G=\{z \in D : d(z,S_{\gamma})<\frac{R}{2}\vert \sin \gamma \vert\},$
the result follows.
\end{proof}

\begin{lemma}\label{L4}
Let $C_{i}$ denote the cusp which is adjacent to $S_{\gamma}$ and $B_{i}(R)$.
Let $H$ be the half-plane as in Lemma~\ref{L3}. Then
\begin{equation*}
\int_{C_{i}}dx \int_{D} dy \, p(x,y;t) = \int_{C_{i}}dx \int_{H} dy \, p(x,y;t)
+O(e^{-R^{2}(\sin \gamma)^{2}/(32t)}), \, t \downarrow 0.
\end{equation*}
\end{lemma}

\begin{proof}
Using Lemma~\ref{L} with $\tilde{D}=D$, $F=H$, $E=C_{i}$ and $G=\{z \in D : d(z,C_{i})<\frac{R}{2}\vert \sin \gamma \vert\},$
the result follows.
\end{proof}


\section{The contribution to the heat content from points close to a vertex.}\label{S3}
In this section we approximate $u_{D}$ by $u_{W_{j}}$. We then compute the contribution to the heat
content of $D$ from a sector, $B_{j}(R)$ with corresponding angle $\gamma_{j}=\beta$. We also compute
the contribution to the heat content of $D$ from two disjoint wedges whose boundaries intersect in a vertex
of $\partial D$.

Firstly, we define
\begin{equation}\label{e30}
\mathcal{V}_{\beta}(t;R) = \int_{B_{j}(R)} dx \int_{W_{j}} dy \, p(x,y;t).
\end{equation}
\begin{lemma}\label{L3.1}
For $\beta \in (0,\pi)\cup(\pi,2\pi)$,
\begin{align*}
\mathcal{V}_{\beta}(t;R)
&=\frac{\beta R^{2}}{2}-\frac{2R}{\sqrt{\pi}}t^{1/2}+\left(\frac{1}{\pi}+\left(1- \frac{\beta}{\pi}\right)\cot \beta\right)t\\
&-(4 \pi t)^{-1/2}\int_{0}^{R\vert \sin \beta \vert}dx \,\left(-2Rx + R^{2}\arcsin \left(\frac{x}{R}\right)
+x\sqrt{R^{2}-x^{2}}\right)e^{-x^{2}/(4t)}\\
&+O(t e^{-R^{2}(\sin \beta)^{2}/(8t)}),\, t \downarrow 0.\\
\end{align*}
\end{lemma}
\begin{proof}
Changing to polar coordinates with $x=(r_{1}\cos \theta_{1},r_{1}\sin\theta_{1})$ and $y=(r_{2}\cos \theta_{2},r_{2}\sin\theta_{2})$
in \eqref{e30} gives that
\begin{align*}
&\mathcal{V}_{\beta}(t;R)\\
&=(4 \pi t)^{-1} \int_{0}^{\beta} d \theta_{1} \int_{0}^{\beta} d \theta_{2} \int_{0}^{R} dr_{1}
\int_{0}^{\infty} dr_{2} (r_{1} r_{2}) e^{-(r_{1}^{2} + r_{2}^{2})/(4t) + 2r_{1} r_{2}A/(4t)},
\end{align*}
where $A=\cos (\theta_{1}-\theta_{2})$. The change of variable $r_{2} - r_{1} A = \rho$ gives that
\begin{align*}
\mathcal{V}_{\beta}(t;R)
&=(4 \pi t)^{-1} \int_{0}^{\beta} d \theta_{1} \int_{0}^{\beta} d \theta_{2} \int_{0}^{R} dr_{1}
\int_{0}^{\infty} dr_{2} (r_{1} r_{2}) e^{-(r_{2}-r_{1}A)^{2}/(4t)-r_{1}^{2}(1-A^{2})/(4t)}\\
&=(4 \pi t)^{-1} \int_{0}^{\beta} d \theta_{1} \int_{0}^{\beta} d \theta_{2} \int_{0}^{R} r \, dr
\int_{-Ar}^{\infty} d \rho \, (\rho + Ar) e^{-\rho^{2}/(4t) - r^{2}(1-A^{2})/(4t)}\\
&= I_{1} + I_{2}.
\end{align*}
We have that
\begin{align}
I_{1}
&=(4 \pi t)^{-1} \int_{0}^{\beta} d \theta_{1} \int_{0}^{\beta} d \theta_{2} \int_{0}^{R} r \, dr
\int_{-Ar}^{\infty} d \rho \, \rho e^{-\rho^{2}/(4t) - r^{2}(1-A^{2})/(4t)} \notag\\
&=\frac{\beta^{2}}{\pi} t (1-e^{-R^{2}/(4t)}), \label{e30a}
\end{align}
and
\begin{align*}
I_{2}
&=(4 \pi t)^{-1} \int_{0}^{\beta} d \theta_{1} \int_{0}^{\beta} d \theta_{2} \int_{0}^{R} dr \, Ar^{2}
\int_{-Ar}^{\infty} d \rho \, e^{-\rho^{2}/(4t) - r^{2}(1-A^{2})/(4t)}\\
&=(4 \pi t)^{-1} \int_{0}^{\beta} d \theta_{1} \int_{0}^{\beta} d \theta_{2} \int_{0}^{R} dr \, Ar^{2}
e^{-r^{2}(1-A^{2})/(4t)}\int_{0}^{\infty} d \rho \, e^{-\rho^{2}/(4t)}\\
&\ \ \ +(4 \pi t)^{-1} \int_{0}^{\beta} d \theta_{1} \int_{0}^{\beta} d \theta_{2} \int_{0}^{R} dr \, A^{2}r^{2}
e^{-r^{2}(1-A^{2})/(4t)}\int_{0}^{r} d \rho \, e^{-A^{2}\rho^{2}/(4t)}\\
&=(4 \pi t)^{-1} \int_{0}^{\beta} d \theta_{1} \int_{0}^{\beta} d \theta_{2} \int_{0}^{R} dr \, Ar^{2}
e^{-r^{2}(1-A^{2})/(4t)}\int_{0}^{\infty} d \rho \, e^{-\rho^{2}/(4t)}\\
&\ \ \ +(4 \pi t)^{-1} \int_{0}^{\beta} d \theta_{1} \int_{0}^{\beta} d \theta_{2} \int_{0}^{R} dr \, A^{2}r^{2}
e^{-r^{2}(1-A^{2})/(4t)}\int_{0}^{\infty} d \rho \, e^{-A^{2}\rho^{2}/(4t)}\\
&\ \ \ -(4 \pi t)^{-1} \int_{0}^{\beta} d \theta_{1} \int_{0}^{\beta} d \theta_{2} \int_{0}^{R} dr \, A^{2}r^{2}
e^{-r^{2}(1-A^{2})/(4t)}\int_{r}^{\infty} d \rho \, e^{-A^{2}\rho^{2}/(4t)}\\
&=I_{3} + I_{4} + I_{5}.
\end{align*}
Via the change of variables $\theta_{1}-\theta_{2}=-\eta$ and integrating by parts with respect
to $\theta$, we obtain
\begin{align}
I_{3}
&=(4 \pi t)^{-1} \int_{0}^{\beta} d \theta_{1} \int_{0}^{\beta} d \theta_{2} \int_{0}^{R} dr \, Ar^{2}
e^{-r^{2}(1-A^{2})/(4t)}\int_{0}^{\infty} d \rho \, e^{-\rho^{2}/(4t)}\notag\\
&=(4 \pi t)^{-1/2} \int_{0}^{\beta} d \theta_{1} \int_{\theta_{1}}^{\beta} d \theta_{2} \int_{0}^{R} dr \, r^{2}
\cos (\theta_{1}-\theta_{2}) \, e^{-r^{2}(\sin (\theta_{1}-\theta_{2})^{2})/(4t)}\notag\\
&=(4 \pi t)^{-1/2} \int_{0}^{\beta} d \theta_{1}\int_{0}^{\beta -\theta_{1}} d \eta \int_{0}^{R} dr \,
r^{2} \cos \eta \, e^{-r^{2}(\sin \eta)^{2}/(4t)}\notag\\
&=(4 \pi t)^{-1/2} \int_{0}^{\beta} d \theta \int_{0}^{\theta} d \eta \int_{0}^{R} dr \,
r^{2} \cos \eta \, e^{-r^{2}(\sin \eta)^{2}/(4t)}\notag\\
&=\beta(4 \pi t)^{-1/2} \int_{0}^{\beta}d\eta \int_{0}^{R} dr \, r^{2} \cos \eta \,
e^{-r^{2}(\sin \eta)^{2}/(4t)}\label{e30b}\\
&\ \ \ -(4 \pi t)^{-1/2}\int_{0}^{\beta} d \theta \, \theta
\int_{0}^{R} dr \, r^{2} \cos \theta \, e^{-r^{2}(\sin \theta)^{2}/(4t)}.\notag
\end{align}
Similarly,
\begin{align}
I_{4}&=\beta(4 \pi t)^{-1/2} \int_{0}^{\beta}d\eta \int_{0}^{R} dr \, r^{2} \vert \cos \eta \vert
\, e^{-r^{2}(\sin \eta)^{2}/(4t)}\label{e30c}\\
&\ \ \ - (4 \pi t)^{-1/2}\int_{0}^{\beta} d \theta \, \theta
\int_{0}^{R} dr \, r^{2} \vert \cos \theta \vert \, e^{-r^{2}(\sin \theta)^{2}/(4t)}.\notag
\end{align}
We first compute $I_{3}+I_{4}$ and then we deal with $I_{5}$.
By \eqref{e30b} and \eqref{e30c}, we have that
\begin{equation}\label{e30d}
I_{3}+I_{4}
=2(4 \pi t)^{-1/2} \int_{0}^{R} dr \, r^{2}  \int_{B} d\eta \, (\beta - \eta) \cos \eta
\, e^{-r^{2}(\sin \eta)^{2}/(4t)},
\end{equation}
where $B=[0,\beta] \cap ([0,\frac{\pi}{2}] \cup [\frac{3\pi}{2},2\pi])$.
First suppose $\beta \in [0,\frac{\pi}{2}]$, then by \eqref{e30d} we have that
\begin{align}
&I_{3}+I_{4}\notag\\
&=2(4 \pi t)^{-1/2} \int_{0}^{R} dr \, r^{2}  \int_{0}^{\beta} d\eta \, (\beta - \eta) \cos \eta
\, e^{-r^{2}(\sin \eta)^{2}/(4t)}\notag\\
&=2(4 \pi t)^{-1/2} \int_{0}^{R} dr \, r^{2}  \int_{0}^{\sin \beta} d\psi \, (\beta - \arcsin \psi)
\, e^{-r^{2}\psi^{2}/(4t)}\notag\\
&=2\beta(4 \pi t)^{-1/2} \int_{0}^{R} dr \, r^{2}  \int_{0}^{\sin \beta} d\psi \, e^{-r^{2}\psi^{2}/(4t)}
-2(4 \pi t)^{-1/2} \int_{0}^{R} dr \, r^{2}  \int_{0}^{\sin \beta} d\psi \, \psi e^{-r^{2}\psi^{2}/(4t)}\notag\\
&\ \ \ -2(4 \pi t)^{-1/2} \int_{0}^{R} dr \, r^{2}  \int_{0}^{\sin \beta} d\psi \, (\arcsin \psi - \psi) e^{-r^{2}\psi^{2}/(4t)}\notag
\end{align}
\begin{align}
&=2\beta(4 \pi t)^{-1/2} \int_{0}^{R} dr \, r^{2}  \int_{0}^{\infty} d\psi \, e^{-r^{2}\psi^{2}/(4t)}
-2\beta(4 \pi t)^{-1/2} \int_{0}^{R} dr \, r^{2}  \int_{\sin \beta}^{\infty} d\psi \, e^{-r^{2}\psi^{2}/(4t)}\notag\\
&\ \ \ +4t(4\pi t)^{-1/2}\int_{0}^{R}dr \, (e^{-r^{2}(\sin \beta)^{2}/(4t)} -1)\notag\\
&\ \ \ -2(4 \pi t)^{-1/2} \int_{0}^{\sin \beta} d\psi \, (\arcsin \psi - \psi) \int_{0}^{R} dr \, r^{2} e^{-r^{2}\psi^{2}/(4t)}\notag\\
&=\frac{\beta R^{2}}{2}-2\beta(4 \pi t)^{-1/2} \int_{\sin \beta}^{\infty} d\psi \int_{0}^{\infty} dr \, r^{2} \, e^{-r^{2}\psi^{2}/(4t)}
-\frac{2R}{\sqrt{\pi}}t^{1/2} + \frac{2}{\sin \beta}t\notag\\
&\ \ \ -2(4 \pi t)^{-1/2} \int_{0}^{\sin \beta} d\psi \, (\arcsin \psi - \psi) \int_{0}^{\infty} dr \, r^{2} e^{-r^{2}\psi^{2}/(4t)}\notag\\
&\ \ \ +2(4 \pi t)^{-1/2} \int_{0}^{\sin \beta} d\psi \, (\arcsin \psi - \psi) \int_{R}^{\infty} dr \, r^{2} e^{-r^{2}\psi^{2}/(4t)}
+O(te^{-R^{2}(\sin \beta)^{2}/(8t)})\notag\\
&=\frac{\beta R^{2}}{2}-\frac{2R}{\sqrt{\pi}}t^{1/2} - 2\beta t\int_{\sin \beta}^{\infty} \frac{d\psi}{\psi^{3}} + \frac{2}{\sin \beta}t
-2t\int_{0}^{\sin \beta} d\psi \, \left(\frac{\arcsin \psi - \psi}{\psi^{3}}\right)\notag\\
&\ \ \ +2(4 \pi t)^{-1/2} \int_{0}^{\sin \beta} d\psi \, (\arcsin \psi - \psi) \int_{R}^{\infty} dr \, r^{2} e^{-r^{2}\psi^{2}/(4t)}
+O(te^{-R^{2}(\sin \beta)^{2}/(8t)})\notag\\
&=\frac{\beta R^{2}}{2}-\frac{2R}{\sqrt{\pi}}t^{1/2} + (\cot \beta)t
+2(4 \pi t)^{-1/2} \int_{0}^{\sin \beta} d\psi \, (\arcsin \psi - \psi) \int_{R}^{\infty} dr \, r^{2} e^{-r^{2}\psi^{2}/(4t)}\label{e30e}\\
&\ \ \ +O(te^{-R^{2}(\sin \beta)^{2}/(8t)}),\, t \downarrow 0.\notag
\end{align}
Similarly, for $\beta \in [\frac{\pi}{2},\frac{3\pi}{2}]$, we have that
\begin{align}
&I_{3}+I_{4}\notag\\
&=2(4 \pi t)^{-1/2} \int_{0}^{R} dr \, r^{2}  \int_{0}^{\frac{\pi}{2}} d\eta \, (\beta - \eta) \cos \eta
\, e^{-r^{2}(\sin \eta)^{2}/(4t)}\notag\\
&=2(4 \pi t)^{-1/2} \int_{0}^{R} dr \, r^{2}  \int_{0}^{1} d\psi \, (\beta - \arcsin \psi)
\, e^{-r^{2}\psi^{2}/(4t)}\notag\\
&=2\beta(4 \pi t)^{-1/2} \int_{0}^{R} dr \, r^{2}  \int_{0}^{1} d\psi \, e^{-r^{2}\psi^{2}/(4t)}
-2(4 \pi t)^{-1/2} \int_{0}^{R} dr \, r^{2}  \int_{0}^{1} d\psi \, \psi e^{-r^{2}\psi^{2}/(4t)}\notag\\
&\ \ \ -2(4 \pi t)^{-1/2} \int_{0}^{R} dr \, r^{2}  \int_{0}^{1} d\psi \, (\arcsin \psi - \psi) e^{-r^{2}\psi^{2}/(4t)}\notag\\
&=2\beta(4 \pi t)^{-1/2} \int_{0}^{R} dr \, r^{2}  \int_{0}^{\infty} d\psi \, e^{-r^{2}\psi^{2}/(4t)}
-2\beta(4 \pi t)^{-1/2} \int_{0}^{R} dr \, r^{2}  \int_{1}^{\infty} d\psi \, e^{-r^{2}\psi^{2}/(4t)}\notag\\
&\ \ \ +4t(4\pi t)^{-1/2}\int_{0}^{R}dr \, (e^{-r^{2}/(4t)} -1)\notag\\
&\ \ \ -2(4 \pi t)^{-1/2} \int_{0}^{1} d\psi \, (\arcsin \psi - \psi) \int_{0}^{R} dr \, r^{2} e^{-r^{2}\psi^{2}/(4t)}\notag\\
&=\frac{\beta R^{2}}{2}-2\beta t \int_{1}^{\infty} \frac{d\psi}{\psi^{3}}
-\frac{2R}{\sqrt{\pi}}t^{1/2} + 2t -2t \int_{0}^{1} d\psi \, \left(\frac{\arcsin \psi - \psi}{\psi^{3}}\right)\notag\\
&\ \ \ +2(4 \pi t)^{-1/2} \int_{0}^{1} d\psi \, (\arcsin \psi - \psi) \int_{R}^{\infty} dr \, r^{2} e^{-r^{2}\psi^{2}/(4t)}
+O(te^{-R^{2}/(8t)})\notag\\
&=\frac{\beta R^{2}}{2}-\frac{2R}{\sqrt{\pi}}t^{1/2} + \left(\frac{\pi}{2} - \beta\right)t
+2(4 \pi t)^{-1/2} \int_{0}^{\vert \sin \beta \vert} d\psi \, (\arcsin \psi - \psi) \int_{R}^{\infty} dr \, r^{2} e^{-r^{2}\psi^{2}/(4t)}\label{e30f}\\
&\ \ \ +O(te^{-R^{2}(\sin \beta)^{2}/(8t)}),\, t \downarrow 0.\notag
\end{align}
In order to have \eqref{e30f} in the same format as \eqref{e30e}, we replaced 1 by $\vert \sin \beta \vert$
in \eqref{e30f}. For $\beta \in [\frac{3\pi}{2},2\pi]$, we have that
\begin{align*}
I_{3}+I_{4}
&=2(4 \pi t)^{-1/2} \int_{0}^{R} dr \, r^{2}  \int_{0}^{\frac{\pi}{2}} d\eta \, (\beta - \eta) \cos \eta
\, e^{-r^{2}(\sin \eta)^{2}/(4t)}\\
&\ \ \ +2(4 \pi t)^{-1/2} \int_{0}^{R} dr \, r^{2}  \int_{\frac{3\pi}{2}}^{\beta} d\eta \, (\beta - \eta) \cos \eta
\, e^{-r^{2}(\sin \eta)^{2}/(4t)}\\
&=J_{1}+J_{2}.
\end{align*}
Now by \eqref{e30f},
\begin{align}
J_{1}&=\frac{\beta R^{2}}{2}-\frac{2R}{\sqrt{\pi}}t^{1/2} + \left(\frac{\pi}{2} - \beta\right)t
+2(4 \pi t)^{-1/2} \int_{0}^{\vert \sin \beta \vert} d\psi \, (\arcsin \psi - \psi) \int_{R}^{\infty} dr \, r^{2} e^{-r^{2}\psi^{2}/(4t)}\label{e30g}\\
&\ \ \ +O(te^{-R^{2}(\sin \beta)^{2}/(8t)}),\, t \downarrow 0.\notag
\end{align}
Via the change of variables $\eta'=2\pi - \eta$, we also have that
\begin{align}
J_{2}&=2(4 \pi t)^{-1/2} \int_{\frac{3\pi}{2}}^{\beta} d\eta  \, (\beta - \eta) \cos \eta \int_{0}^{R} dr \, r^{2}
\, e^{-r^{2}(\sin \eta)^{2}/(4t)}\notag\\
&=2(4 \pi t)^{-1/2} \int_{2\pi-\beta}^{\frac{\pi}{2}} d\eta  \, (\beta + \eta - 2\pi) \cos \eta \int_{0}^{R} dr \, r^{2}
\, e^{-r^{2}(\sin \eta)^{2}/(4t)}\notag\\
&=2t\int_{2\pi-\beta}^{\frac{\pi}{2}} d\eta  \, (\beta + \eta - 2\pi) \frac{\cos \eta}{(\sin \eta)^{3}}\notag\\
&\ \ \ -2(4 \pi t)^{-1/2} \int_{2\pi-\beta}^{\frac{\pi}{2}} d\eta  \, (\beta + \eta - 2\pi) \cos \eta \int_{R}^{\infty} dr \, r^{2}
\, e^{-r^{2}(\sin \eta)^{2}/(4t)}\notag\\
&=\left(\frac{3\pi}{2}-\beta - \cot\beta\right)t + O(te^{-R^{2}(\sin \beta)^{2}/(8t)}),\, t \downarrow 0.\label{e30h}
\end{align}
Hence by \eqref{e30g} and \eqref{e30h}, we see that for $\beta \in [\frac{3\pi}{2},2\pi]$ and $t \downarrow 0$,
\begin{align}
I_{3}+I_{4}
&=\frac{\beta R^{2}}{2}-\frac{2R}{\sqrt{\pi}}t^{1/2} +(2\pi - 2\beta -\cot\beta)t
\label{e30i}\\
&\ \ \ +2(4 \pi t)^{-1/2} \int_{0}^{\vert \sin \beta \vert} d\psi \, (\arcsin \psi - \psi) \int_{R}^{\infty} dr \, r^{2} e^{-r^{2}\psi^{2}/(4t)}
+O(te^{-R^{2}(\sin \beta)^{2}/(8t)}).\notag
\end{align}
It remains to compute $I_{5}$. Via the change of variables $\rho=r\rho'$, we see that
\begin{align}
I_{5}
&=-(4 \pi t)^{-1} \int_{0}^{\beta} d \theta_{1} \int_{0}^{\beta} d \theta_{2} \int_{0}^{R} dr \, A^{2}r^{2}
e^{-r^{2}(1-A^{2})/(4t)}\int_{r}^{\infty} d \rho \, e^{-A^{2}\rho^{2}/(4t)}\notag\\
&=-(4 \pi t)^{-1} \int_{0}^{\beta} d \theta_{1} \int_{0}^{\beta} d \theta_{2} \int_{0}^{R} dr \, A^{2}r^{3}
e^{-r^{2}(1-A^{2})/(4t)}\int_{1}^{\infty} d \rho \, e^{-r^{2}A^{2}\rho^{2}/(4t)}\notag\\
&=-(8 \pi t)^{-1} \int_{0}^{\beta} d \theta_{1} \int_{0}^{\beta} d \theta_{2} \int_{0}^{R^{2}} dr \, A^{2}r
\int_{1}^{\infty} d \rho \, e^{-r(A^{2}\rho^{2}+1-A^{2})/(4t)}\notag\\
&=-(8 \pi t)^{-1} \int_{0}^{\beta} d \theta_{1} \int_{0}^{\beta} d \theta_{2} \int_{1}^{\infty} d \rho
\int_{0}^{\infty} dr \, A^{2}r e^{-r(A^{2}\rho^{2}+1-A^{2})/(4t)}\notag\\
&\ \ \ +(8 \pi t)^{-1} \int_{0}^{\beta} d \theta_{1} \int_{0}^{\beta} d \theta_{2} \int_{1}^{\infty} d \rho
\int_{R^{2}}^{\infty} dr \, A^{2}r e^{-r(A^{2}\rho^{2}+1-A^{2})/(4t)}\notag\\
&=-\frac{2t}{\pi}\int_{0}^{\beta} d \theta_{1} \int_{0}^{\beta} d \theta_{2} \int_{1}^{\infty} d \rho
\frac{A^{2}}{(A^{2}\rho^{2}+1-A^{2})^{2}} + O(te^{-R^{2}/(4t)}),\, t \downarrow 0.\label{e31a}
\end{align}
By 2.173 (1) and 2.172 in \cite{GR}, we see that
\begin{equation}\label{e33a}
\int_{1}^{\infty} \frac{d\rho}{(A^{2}\rho^{2}+1-A^{2})^{2}}= \frac{-1}{2(1-A^{2})}
+\frac{1}{2\vert A \vert (1-A^{2})^{3/2}}\arctan \left(\frac{\sqrt{1-A^{2}}}{\vert A \vert}\right).
\end{equation}
Therefore by \eqref{e31a}, \eqref{e33a} and the change of variables $\theta_{2}-\theta_{1}=\sigma$,
we obtain that
\begin{align*}
&-\frac{2t}{\pi}\int_{0}^{\beta} d \theta_{1} \int_{0}^{\beta} d \theta_{2} \int_{1}^{\infty} d \rho
\frac{A^{2}}{(A^{2}\rho^{2}+1-A^{2})^{2}}\notag\\
&=-\frac{2t}{\pi}\int_{0}^{\beta} d \theta \int_{0}^{\theta} d \sigma
\left(-(\cot \sigma)^{2} + \frac{\cos \sigma}{(\sin \sigma)^{3}}\arctan (\tan \sigma)\right).
\end{align*}
Thus
\begin{equation*}
I_{5}=-\frac{2t}{\pi}\int_{0}^{\beta} d \theta \int_{0}^{\theta} d \sigma
\left(-(\cot \sigma)^{2} + \frac{\cos \sigma}{(\sin \sigma)^{3}}\arctan (\tan \sigma)\right)
+O(t e^{-R^{2}/(4t)}),\, t \downarrow 0.
\end{equation*}
We note that
\begin{equation*}
\arctan (\tan \sigma) = \sigma + U(\sigma),
\end{equation*}
where
\[U(\sigma) = \begin{cases}
0, &\text{if $\sigma \in (0,\frac{\pi}{2})$;}\\
-\pi, &\text{if $\sigma \in (\frac{\pi}{2},\frac{3\pi}{2})$;}\\
-2\pi, &\text{if $\sigma \in (\frac{3\pi}{2},2\pi)$.}
\end{cases}\]
Hence it is necessary to compute
\begin{align*}
-\frac{2t}{\pi}\int_{0}^{\beta} d \theta \int_{0}^{\theta} d \sigma
\left(-(\cot \sigma)^{2} + \frac{\sigma \cos \sigma}{(\sin \sigma)^{3}}\right)
=&-\frac{2t}{\pi}\int_{0}^{\beta} d \theta \left(\frac{\cot \theta}{2} + \theta -
\frac{\theta}{2(\sin \theta)^{2}}\right)\notag\\
=&\left(-\frac{\beta^{2}}{\pi}- \frac{\beta}{\pi}\cot \beta + \frac{1}{\pi}\right)t.
\end{align*}
If $\beta \in (\frac{\pi}{2},\frac{3\pi}{2})$, then
\begin{equation*}
-\frac{2t}{\pi}\int_{0}^{\beta} d \theta \int_{0}^{\theta} d \sigma \, U(\sigma) \frac{\cos \sigma}{(\sin \sigma)^{3}}
=2t\int_{\frac{\pi}{2}}^{\beta}d\theta \int_{\frac{\pi}{2}}^{\theta}
d\sigma \frac{\cos \sigma}{(\sin \sigma)^{3}} =\left(\beta + \cot \beta - \frac{\pi}{2}\right)t,
\end{equation*}
and if $\beta \in (\frac{3\pi}{2},2\pi)$, then
\begin{align*}
&-\frac{2t}{\pi}\int_{0}^{\beta} d \theta \int_{0}^{\theta} d \sigma \, U(\sigma) \frac{\cos \sigma}{(\sin \sigma)^{3}}\notag\\
&= 2t\int_{\frac{\pi}{2}}^{\frac{3\pi}{2}}d\theta \int_{\frac{\pi}{2}}^{\theta} d\sigma \frac{\cos \sigma}{(\sin \sigma)^{3}}
+ 2t\int_{\frac{3\pi}{2}}^{\beta}d\theta \int_{\frac{\pi}{2}}^{\frac{3\pi}{2}}d\sigma \frac{\cos \sigma}{(\sin \sigma)^{3}}
+4t\int_{\frac{3\pi}{2}}^{\beta}d\theta \int_{\frac{3\pi}{2}}^{\theta}d\sigma \frac{\cos \sigma}{(\sin \sigma)^{3}}\notag\\
&=(2\beta + 2\cot\beta -2\pi)t.
\end{align*}
Hence
\begin{equation}\label{e34f}
I_{5}=\begin{cases}
\left(-\frac{\beta^{2}}{\pi}- \frac{\beta}{\pi}\cot \beta + \frac{1}{\pi}\right)t,
&\text{if $\beta \in (0,\frac{\pi}{2})$;}\\
\left(-\frac{\beta^{2}}{\pi}+ \left(1- \frac{\beta}{\pi}\right)\cot \beta +\beta - \frac{\pi}{2} + \frac{1}{\pi}\right)t,
&\text{if $\beta \in (\frac{\pi}{2},\frac{3\pi}{2})$;}\\
\left(-\frac{\beta^{2}}{\pi}+ \left(1- \frac{\beta}{\pi}\right)\cot \beta +\cot\beta + 2\beta -2\pi + \frac{1}{\pi}\right)t,
&\text{if $\beta \in (\frac{3\pi}{2},2\pi)$;}\\
\end{cases}
\end{equation}
Therefore, by combining \eqref{e30a}, \eqref{e34f} and \eqref{e30e}, \eqref{e30f}, \eqref{e30i} respectively,
we obtain that, as $t \downarrow 0$,
\begin{align*}
\mathcal{V}_{\beta}(t;R)
&=\frac{\beta R^{2}}{2}-\frac{2R}{\sqrt{\pi}}t^{1/2}+\left(\frac{1}{\pi}+\left(1- \frac{\beta}{\pi}\right)\cot \beta\right)t\\
&+2(4 \pi t)^{-1/2} \int_{0}^{\vert \sin \beta \vert}d\psi \, (\arcsin \psi-\psi) \int_{R}^{\infty} dr \, r^{2}\, e^{-r^{2}\psi^{2}/(4t)}
+O(t e^{-R^{2}(\sin \beta)^{2}/(8t)}).\\
\end{align*}
Via the change of variables $\rho = r\psi$ and integrating by parts with respect to $\psi$, we have
\begin{align}
&2(4 \pi t)^{-1/2} \int_{0}^{\vert \sin \beta \vert}d\psi \, (\arcsin \psi-\psi) \int_{R}^{\infty} dr \, r^{2}\, e^{-r^{2}\psi^{2}/(4t)}\notag\\
&=2(4 \pi t)^{-1/2} \int_{0}^{\vert \sin \beta \vert}d\psi \, \left(\frac{\arcsin \psi-\psi}{\psi^{3}}\right)\int_{R\psi}^{\infty}d\rho \, \rho^{2}
e^{-\rho^{2}/(4t)}\notag\\
&=-(4 \pi t)^{-1/2}R^{3}\int_{0}^{\vert \sin \beta \vert}d\psi \,(\arcsin \psi-\psi)e^{-R^{2}\psi^{2}/(4t)}\notag\\
&\ \ \ +(4 \pi t)^{-1/2}\int_{0}^{\vert \sin \beta \vert}\frac{d\psi}{\psi^{2}} \,\left(\frac{1-\sqrt{1-\psi^{2}}}{\sqrt{1-\psi^{2}}}\right)
\int_{R\psi}^{\infty}d\rho \, \rho^{2}e^{-\rho^{2}/(4t)}+O(t e^{-R^{2}(\sin \beta)^{2}/(8t)})\notag\\
&=-(4 \pi t)^{-1/2}R^{3}\int_{0}^{\vert \sin \beta \vert}d\psi \,(\arcsin \psi-\psi)e^{-R^{2}\psi^{2}/(4t)}\label{e34g}\\
&\ \ \ +(4 \pi t)^{-1/2}\int_{0}^{\vert \sin \beta \vert}d\psi \, \left(\frac{1}{\sqrt{1-\psi^{2}}}-\frac{1}{1+\sqrt{1-\psi^{2}}}\right)
\int_{R\psi}^{\infty}d\rho \, \rho^{2}e^{-\rho^{2}/(4t)}\notag\\
&\ \ \ +O(t e^{-R^{2}(\sin \beta)^{2}/(8t)}),\, t \downarrow 0.\notag
\end{align}
Again integrating by parts with respect to $\psi$, we see that
\begin{align}
&(4 \pi t)^{-1/2}\int_{0}^{\vert \sin \beta \vert}d\psi \, \left(\frac{1}{\sqrt{1-\psi^{2}}}-\frac{1}{1+\sqrt{1-\psi^{2}}}\right)
\int_{R\psi}^{\infty}d\rho \, \rho^{2}e^{-\rho^{2}/(4t)}\notag\\
&=(4 \pi t)^{-1/2}R^{3}\int_{0}^{\vert \sin \beta \vert}d\psi \,(-\psi\sqrt{1-\psi^{2}}+\psi)e^{-R^{2}\psi^{2}/(4t)}
+O(t e^{-R^{2}(\sin \beta)^{2}/(8t)}),\, t \downarrow 0.\label{e34h}
\end{align}
Hence by \eqref{e34g}, \eqref{e34h} and the change of variables $\psi=\frac{x}{R}$, we have that
\begin{align}\label{e38}
&2(4 \pi t)^{-1/2} \int_{0}^{\vert \sin \beta \vert}d\psi \, (\arcsin \psi-\psi) \int_{R}^{\infty} dr \, r^{2}\, e^{-r^{2}\psi^{2}/(4t)}\notag\\
&=-(4 \pi t)^{-1/2}\int_{0}^{R\vert \sin \beta \vert}dx \,\left(R^{2}\arcsin \left(\frac{x}{R}\right)-Rx\right)e^{-x^{2}/(4t)}\notag\\
&\ \ \ +(4 \pi t)^{-1/2}\int_{0}^{R\vert \sin \beta \vert}dx \, \left(-x\sqrt{R^{2}-x^{2}}+Rx\right)
e^{-x^{2}/(4t)} +O(t e^{-R^{2}(\sin \beta)^{2}/(8t)})\notag\\
&=-(4 \pi t)^{-1/2}\int_{0}^{R\vert \sin \beta \vert}dx \,\left(-2Rx + R^{2}\arcsin \left(\frac{x}{R}\right)
+x\sqrt{R^{2}-x^{2}}\right)e^{-x^{2}/(4t)}\\
&\ \ \ +O(t e^{-R^{2}(\sin \beta)^{2}/(8t)}),\, t \downarrow 0.\notag
\end{align}
This completes the proof of Lemma~\ref{L3.1}.
\end{proof}
In Section~\ref{SS6}, we deal with the remaining integral in \eqref{e38}.
\begin{lemma}\label{L3.2}
Let $W_{1}, W_{2}$ be two disjoint wedges in $\R^{2}$ with corresponding angles
$\gamma_{1}, \gamma_{2}$ respectively such that $\partial{W}_{1} \cap \partial{W}_{2}
 = \{V_{i}\}$ for some $i \in \{1, \dots, n\}$. Let $\alpha$ denote the angle between
$W_{1}$ and $W_{2}$ such that $0<\alpha\leq \pi$. Then
\begin{equation*}
\int_{W_{1}} dx \int_{W_{2}} dy \, p(x,y;t)=k(\alpha, \gamma_{1}, \gamma_{2}) t,
\end{equation*}
where $k(\alpha, \gamma_{1}, \gamma_{2})$ is as defined in \eqref{e11b}.
\end{lemma}
\begin{proof}
By changing coordinates to polar coordinates, as in the proof of Lemma~\ref{L3.1}, we have that
\begin{align}
&\int_{W_{1}} dx \int_{W_{2}} dy \, p(x,y;t)\notag\\
&=(4 \pi t)^{-1} \int_{0}^{\gamma_{1}} d \theta_{1} \int_{\gamma_{1}+\alpha}^{\gamma_{1}+\alpha+\gamma_{2}}
d \theta_{2} \int_{0}^{\infty} dr_{1}\int_{0}^{\infty} dr_{2} (r_{1} r_{2}) \, e^{-(r_{1}^{2} + r_{2}^{2})/(4t)
+ 2r_{1} r_{2} \cos(\theta_{1} - \theta_{2})/(4t)}\notag\\
&=\frac{4t}{\pi}\int_{0}^{\gamma_{1}} d \theta_{1} \int_{\gamma_{1}+\alpha}^{\gamma_{1}+\alpha+\gamma_{2}}
d \theta_{2} \int_{0}^{\infty} dr_{1}\int_{0}^{\infty} dr_{2} (r_{1} r_{2}) \, e^{-(r_{1}^{2} + r_{2}^{2})
+ 2r_{1} r_{2} \cos(\theta_{1} - \theta_{2})}\notag\\
&=\frac{4t}{\pi}\int_{0}^{\gamma_{1}} d \theta_{1} \int_{\gamma_{1}+\alpha}^{\gamma_{1}+\alpha+\gamma_{2}}
d \theta_{2} \int_{0}^{\infty} dr_{1} \, r_{1}^{3} \int_{0}^{\infty} d\rho \, \rho \, e^{-r_{1}^{2}(1+\rho^{2}-
2\rho\cos(\theta_{1} - \theta_{2}))}\notag\\
&=\frac{2t}{\pi}\int_{0}^{\gamma_{1}} d \theta_{1} \int_{\gamma_{1}+\alpha}^{\gamma_{1}+\alpha+\gamma_{2}}
d \theta_{2} \int_{0}^{\infty} d\rho \, \rho \int_{0}^{\infty} dr \, r \, e^{-r(1+\rho^{2}-
2\rho\cos(\theta_{1} - \theta_{2}))}\notag\\
&=\frac{2t}{\pi}\int_{0}^{\gamma_{1}} d \theta_{1} \int_{\gamma_{1}+\alpha}^{\gamma_{1}+\alpha+\gamma_{2}}
d \theta_{2} \int_{0}^{\infty} d\rho \frac{\rho}{(1+\rho^{2}-2\rho\cos(\theta_{1} - \theta_{2}))^{2}}\notag\\
&=\frac{2t}{\pi}\int_{0}^{\gamma_{1}} d \theta_{1} \int_{\gamma_{1}+\alpha-\theta_{1}}^{\gamma_{1}+\alpha+\gamma_{2}-\theta_{1}}
d \varphi \int_{0}^{\infty} d\rho \frac{\rho}{(1+\rho^{2}-2\rho\cos\varphi)^{2}}\notag\\
&=\frac{2t}{\pi}\int_{\alpha}^{\gamma_{1}+\alpha} d \sigma \int_{\sigma}^{\sigma+\gamma_{2}}
d \varphi \int_{0}^{\infty} d\rho \frac{\rho}{(1+\rho^{2}-2\rho\cos\varphi)^{2}}\notag\\
&=\frac{2t}{\pi}\int_{\alpha}^{\gamma_{1}+\alpha} d \sigma \int_{\sigma}^{\sigma+\gamma_{2}}
d \varphi \int_{0}^{\infty} d\rho \left(\frac{\rho-\cos\varphi}{(1+\rho^{2}-2\rho\cos\varphi)^{2}}
+\frac{\cos\varphi}{(1+\rho^{2}-2\rho\cos\varphi)^{2}}\right)\notag\\
&=\frac{\gamma_{1}\gamma_{2}}{\pi}t + \frac{2t}{\pi}\int_{\alpha}^{\gamma_{1}+\alpha} d \sigma
\int_{\sigma}^{\sigma+\gamma_{2}}d \varphi \int_{-\cos\varphi}^{\infty} d\rho \frac{\cos\varphi}{(\rho^{2}+(\sin\varphi)^{2})^{2}}\notag\\
&=\frac{\gamma_{1}\gamma_{2}}{\pi}t + \frac{2t}{\pi}\int_{\alpha}^{\gamma_{1}+\alpha} d \sigma
\int_{\sigma}^{\sigma+\gamma_{2}}d \varphi \frac{\cos\varphi}{\vert \sin\varphi \vert^{3}}\int_{0}^{\infty}
\frac{d\rho}{(\rho^{2}+1)^{2}}\notag\\
&\ \ \ +\frac{2t}{\pi}\int_{\alpha}^{\gamma_{1}+\alpha} d \sigma \int_{\sigma}^{\sigma+\gamma_{2}}d \varphi
\frac{\vert \cos\varphi \vert}{\vert \sin\varphi \vert^{3}}\int_{0}^{\vert \cot \varphi \vert} \frac{d\rho}{(\rho^{2}+1)^{2}}\notag\\
&=\frac{\gamma_{1}\gamma_{2}}{\pi}t + \frac{2t}{\pi}\int_{\alpha}^{\gamma_{1}+\alpha} d \sigma
\int_{\sigma}^{\sigma+\gamma_{2}}d \varphi \frac{\cos\varphi}{\vert \sin\varphi \vert^{3}}\int_{0}^{\frac{\pi}{2}}
d\theta \, (\cos\theta)^{2}\notag\\
&\ \ \ +\frac{2t}{\pi}\int_{\alpha}^{\gamma_{1}+\alpha} d \sigma \int_{\sigma}^{\sigma+\gamma_{2}}d \varphi
\frac{\vert \cos\varphi \vert}{\vert \sin\varphi \vert^{3}}\int_{0}^{\arctan(\vert \cot \varphi \vert)}
d\theta \, (\cos\theta)^{2}\notag\\
&=\frac{\gamma_{1}\gamma_{2}}{\pi}t + \frac{2t}{\pi}\int_{\alpha}^{\gamma_{1}+\alpha} d \sigma
\int_{\sigma}^{\sigma+\gamma_{2}}d \varphi \, \frac{\pi}{4}\frac{\cos\varphi}{\vert \sin\varphi \vert^{3}}\label{e39a}\\
&\ \ \ + \frac{2t}{\pi}\int_{\alpha}^{\gamma_{1}+\alpha} d \sigma \int_{\sigma}^{\sigma+\gamma_{2}}d \varphi
\left(\frac{\sin(2\arctan(\vert \cot \varphi \vert))}{4}+ \frac{\arctan(\vert \cot \varphi \vert)}{2}\right)
\frac{\vert \cos\varphi \vert}{\vert \sin\varphi \vert^{3}}.\notag
\end{align}
We see that
\begin{equation*}
S(\varphi) := \arctan(\vert \cot \varphi \vert)
=\begin{cases}
\vert \varphi - \frac{\pi}{2} \vert, &\text{if $0<\varphi\leq \pi$;}\\
\vert \varphi - \frac{3\pi}{2} \vert, &\text{if $\pi<\varphi\leq 2\pi$.}
\end{cases}
\end{equation*}
Hence
\begin{equation*}
\frac{\sin(2\arctan(\vert \cot \varphi \vert))}{4}\frac{\vert \cos\varphi \vert}{\vert \sin\varphi \vert^{3}}
=\frac{1}{2}\left(\frac{\cos\varphi}{\sin\varphi}\right)^{2}.
\end{equation*}
Therefore \eqref{e39a} becomes
\begin{equation}\label{e39d}
\frac{\gamma_{1}\gamma_{2}}{\pi}t + \frac{2t}{\pi}\int_{\alpha}^{\gamma_{1}+\alpha} d \sigma
\int_{\sigma}^{\sigma+\gamma_{2}}d \varphi \left(\frac{\pi}{4}\frac{\cos\varphi}{\vert \sin\varphi \vert^{3}}
+\frac{1}{2}\left(\frac{\cos\varphi}{\sin\varphi}\right)^{2}+\frac{S(\varphi)\vert \cos\varphi \vert}{2\vert \sin\varphi \vert^{3}}\right).
\end{equation}
By considering the function
\begin{equation*}
f(\varphi)=\frac{\pi}{4}\frac{\cos\varphi}{\vert \sin\varphi \vert^{3}}
+\frac{1}{2}\left(\frac{\cos\varphi}{\sin\varphi}\right)^{2}+\frac{S(\varphi)\vert \cos\varphi \vert}{2\vert \sin\varphi \vert^{3}}
\end{equation*}
for $\varphi$ in each of the four quadrants, we have that
\begin{equation*}
\frac{1}{2}\int_{\sigma}^{\sigma+\gamma_{2}}d\varphi \left((\pi - \varphi)
\frac{\cos\varphi}{(\sin\varphi)^{3}}+\left(\frac{\cos\varphi}{\sin\varphi}\right)^{2}\right)
\end{equation*}
is a primitive for $f(\varphi)$ for all $\varphi \in (0,2\pi)$.
Integrating by parts with respect to $\varphi$, we obtain
\begin{align}
&\frac{1}{2}\int_{\sigma}^{\sigma+\gamma_{2}}d\varphi \left((\pi - \varphi)
\frac{\cos\varphi}{(\sin\varphi)^{3}}+\left(\frac{\cos\varphi}{\sin\varphi}\right)^{2}\right)\notag\\
&=\frac{1}{4}\left(\frac{(\gamma_{2}+\sigma-\pi)}{(\sin(\gamma_{2}+\sigma))^{2}}
-\cot(\gamma_{2}+\sigma) - \frac{(\sigma-\pi)}{(\sin\sigma)^{2}}+\cot\sigma - 2\gamma_{2}\right).\label{e39g}
\end{align}
Hence, by \eqref{e39d}, \eqref{e39g} and integrating by parts with respect to $\sigma$, we have that
\begin{align*}
&\int_{W_{1}} dx \int_{W_{2}} dy \, p(x,y;t)\\
&=\frac{t}{2\pi}\left(-(\gamma_{2}+\gamma_{1}+\alpha-\pi)\cot(\gamma_{2}+\gamma_{1}+\alpha)-(\alpha-\pi)\cot\alpha\right)\\
&\ \ \ +\frac{t}{2\pi}\left((\gamma_{2}+\alpha-\pi)\cot(\gamma_{2}+\alpha)+(\gamma_{1}+\alpha-\pi)\cot(\gamma_{1}+\alpha)\right)
\end{align*}
as required.
\end{proof}

\section{The contribution to the heat content from points close to an edge.}\label{S4}
In this section, we consider points $x \in C(\frac{R}{2}\vert \sin \gamma \vert, R)$.
We partition this region into $n$ rectangles $S_{\gamma}$ and $2n$ cusps $C_{i}$,
each of height $\frac{R}{2}\vert \sin \gamma \vert$. We approximate $u_{D}$ by $u_{H}$,
where $H \subset \R^{2}$ is the half-plane such that $\emptyset \neq \partial S_{\gamma} \cap \partial D \subset \partial H$
and $S_{\gamma} \subset H$, as in Lemma~\ref{L3}.

\subsection{The contribution to the heat content from a rectangle.}\label{SS4}
We first compute the contribution to the heat content of $D$ from a rectangle,
$S_{\gamma}$, of height $\frac{R}{2}\vert \sin \gamma \vert$ and length $L$, where $L \in \R, L>0$.
We have that
\begin{align*}
\int_{H}dy \, p(x,y;t)
&= \int_{-\infty}^{\infty}dy_{2} \int_{0}^{\infty}dy_{1} (4\pi t)^{-1}
e^{-(x_{1}-y_{1})^{2}/(4t)-(x_{2}-y_{2})^{2}/(4t)}\\
&=1-\int_{x_{1}}^{\infty}d\zeta \, (4\pi t)^{-1/2}e^{-\zeta^{2}/(4t)}.\\
\end{align*}
Let $x_{1} \in (0,\frac{R}{2}\vert \sin \gamma \vert)$ and $x_{2}\in (0,L)$. Then,
by integrating by parts with respect to $x_{1}$, we obtain
\begin{align}
&\int_{S_{\gamma}}dx \int_{H}dy \, p(x,y;t)\notag\\
&=\vert S_{\gamma} \vert - \int_{S_{\gamma}}dx \int_{x_{1}}^{\infty}d\zeta \, (4\pi t)^{-1/2}e^{-\zeta^{2}/(4t)}\notag\\
&=\vert S_{\gamma} \vert-(4\pi t)^{-1/2}\int_{0}^{L}dx_{2} \int_{0}^{\frac{R}{2}\vert \sin \gamma \vert}dx_{1}
\int_{x_{1}}^{\infty}d\zeta \, e^{-\zeta^{2}/(4t)}\notag\\
&=\vert S_{\gamma} \vert-(4\pi t)^{-1/2}\int_{0}^{L}dx_{2} \int_{0}^{\infty}dx_{1}
\int_{x_{1}}^{\infty}d\zeta \, e^{-\zeta^{2}/(4t)}+O(t^{1/2} e^{-R^{2}(\sin \gamma)^{2}/(32t)})\notag\\
&=\vert S_{\gamma} \vert - \frac{L}{\sqrt{\pi}}t^{1/2}
+O(t^{1/2} e^{-R^{2}(\sin \gamma)^{2}/(32t)}), \, t \downarrow 0.\label{e40}
\end{align}

\subsection{The contribution to the heat content from a cusp.}\label{SS5}
We now compute the contribution to the heat content of $D$ from a cusp.
Let $C_{i}$ denote the cusp which is adjacent to $S_{\gamma}$ and $B_{i}(R)$. Then
\begin{align*}
\int_{C_{i}}dx \int_{H} dy \, p(x,y;t)
&=\int_{0}^{\frac{R}{2}\vert \sin \gamma \vert} dx \, (R-\sqrt{R^{2}-x^{2}})
\int_{0}^{\infty}dy\, (4\pi t)^{-1/2}e^{-\vert x-y \vert^{2}/(4t)}\\
&=\int_{0}^{\frac{R}{2}\vert \sin \gamma \vert} dx \, (R-\sqrt{R^{2}-x^{2}})
\left(1-\int_{x}^{\infty}d\zeta \, (4\pi t)^{-1/2} e^{-\zeta^{2}/(4t)}\right)\\
&=\vert C_{i} \vert - \int_{0}^{\frac{R}{2}\vert \sin \gamma \vert} dx \, (R-\sqrt{R^{2}-x^{2}})
\int_{x}^{\infty}d\zeta \, (4\pi t)^{-1/2} e^{-\zeta^{2}/(4t)}.\\
\end{align*}
Integrating by parts with respect to $x$, we obtain
\begin{align*}
&- \int_{0}^{\frac{R}{2}\vert \sin \gamma \vert} dx \, (R-\sqrt{R^{2}-x^{2}})
\int_{x}^{\infty}d\zeta \, (4\pi t)^{-1/2} e^{-\zeta^{2}/(4t)}\\
&=-(4\pi t)^{-1/2}\int_{0}^{\frac{R}{2}\vert \sin \gamma \vert}dx \, \left(Rx-\frac{{R}^{2}}{2}\arcsin \left(\frac{x}{R}\right)
 - \frac{x}{2}\sqrt{R^{2}-x^{2}}\right)e^{-x^{2}/(4t)}\\
&\ \ \ +O(e^{-R^{2}(\sin \gamma)^{2}/(32t)}), \, t \downarrow 0.\\
\end{align*}
Hence the contribution from each cusp is
\begin{align*}
&\int_{C_{i}}dx \int_{H} dy \, p(x,y;t)\\
&=\vert C_{i} \vert + (4\pi t)^{-1/2}\int_{0}^{\frac{R}{2}\vert \sin \gamma \vert}dx \, \left(-Rx+\frac{{R}^{2}}{2}\arcsin \left(\frac{x}{R}\right)
 +\frac{x}{2}\sqrt{R^{2}-x^{2}}\right)e^{-x^{2}/(4t)}\\
&\ \ \ +O(e^{-R^{2}(\sin \gamma)^{2}/(32t)}),\, t \downarrow 0.
\end{align*}

\subsection{The $O(t^{3/2})$ terms from the cusp and sector contributions.}\label{SS6}
Finally, we deal with the remaining integrals from the sector and cusp contributions.
Each sector has two neighbouring cusps so we are interested in the following integral
\begin{align}
&2(4\pi t)^{-1/2}\int_{0}^{\frac{R}{2}\vert \sin \gamma \vert}dx \, \left(-Rx+\frac{{R}^{2}}{2}\arcsin \left(\frac{x}{R}\right)
 +\frac{x}{2}\sqrt{R^{2}-x^{2}}\right)e^{-x^{2}/(4t)} \notag\\
&=(4\pi t)^{-1/2}\int_{0}^{\frac{R}{2}\vert \sin \gamma \vert}dx \, \left(-2Rx+{R}^{2}\arcsin \left(\frac{x}{R}\right)
 +x\sqrt{R^{2}-x^{2}}\right)e^{-x^{2}/(4t)}.\label{e50}
\end{align}
By definition of $\gamma$, in \eqref{e7}, we can write the remaining integral from the sector contribution,
\eqref{e38}, as
\begin{align}
&-(4\pi t)^{-1/2}\int_{0}^{R\vert \sin \beta \vert}dx \, \left(-2Rx+{R}^{2}\arcsin \left(\frac{x}{R}\right)
 +x\sqrt{R^{2}-x^{2}}\right)e^{-x^{2}/(4t)}\notag\\
&=-(4\pi t)^{-1/2}\int_{0}^{R\vert \sin \gamma \vert}dx \, \left(-2Rx+{R}^{2}\arcsin \left(\frac{x}{R}\right)
 +x\sqrt{R^{2}-x^{2}}\right)e^{-x^{2}/(4t)}\label{e51}\\
&\ \ \ +O(e^{-R^{2}(\sin \gamma)^{2}/(8t)}),\, t \downarrow 0.\notag
\end{align}
Adding \eqref{e50} and \eqref{e51}, we obtain
\begin{align*}
&(4\pi t)^{-1/2}\int_{0}^{\frac{R}{2}\vert \sin \gamma \vert}dx \, \left(-2Rx+{R}^{2}\arcsin \left(\frac{x}{R}\right)
 +x\sqrt{R^{2}-x^{2}}\right)e^{-x^{2}/(4t)}\\
&-(4\pi t)^{-1/2}\int_{0}^{R\vert \sin \gamma \vert}dx \, \left(-2Rx+{R}^{2}\arcsin \left(\frac{x}{R}\right)
 +x\sqrt{R^{2}-x^{2}}\right)e^{-x^{2}/(4t)}\\
&=-(4\pi t)^{-1/2}\int_{\frac{R}{2}\vert \sin \gamma \vert}^{R\vert \sin \gamma \vert}dx \, \left(-2Rx+{R}^{2}\arcsin \left(\frac{x}{R}\right)
 +x\sqrt{R^{2}-x^{2}}\right)e^{-x^{2}/(4t)}\\
&=O(e^{-R^{2}(\sin \gamma)^{2}/(32t)}),\, t \downarrow 0.\\
\end{align*}
This completes the proof of Theorem~\ref{T1}.

\section{The heat content of a $\pi$-sector in a $\frac{3\pi}{2}$-wedge.}\label{S6}
In this section, we compute the heat content of a $\pi$-sector in a $\frac{3\pi}{2}$-wedge which share one common edge (and vertex).
This is a crucial ingredient in the computation of the heat content of the fractal polyhedron
$D_{s}$, which was constructed in Section~\ref{S1}, and will be used in Section~\ref{S7}.
\begin{lemma}\label{L6.1}
Let $B_{\pi}(R) \subset W_{\frac{3\pi}{2}}$ such that $B_{\pi}(R)$ and $W_{\frac{3\pi}{2}}$ share one common
edge (and vertex). Then, for $t \downarrow 0$,
\begin{align}
&\int_{B_{\pi}(R)} dx \int_{W_{\frac{3\pi}{2}}} dy \, p(x,y;t) \label{e60}\\
&=\frac{\pi R^{2}}{2} -\frac{R}{\sqrt{\pi}}t^{1/2} +(4 \pi t)^{-1/2}\int_{0}^{1}d\psi \, (\arcsin \psi - \psi)
\int_{R}^{\infty} dr \, r^{2}e^{- r^{2}\psi^{2}/(4t)}\notag\\
&\ \ \ + O(te^{-R^{2}/(8t)}).\notag
\end{align}
\end{lemma}
We remark that the coefficient of $t$ is equal to $0$ in this case.

\begin{proof} Similarly to the proof of Lemma \ref{L3.1}, the left-hand side of \eqref{e60} equals
\begin{align}
&(4 \pi t)^{-1} \int_{0}^{\pi} d \theta_{1} \int_{0}^{\frac{3\pi}{2}} d \theta_{2} \int_{0}^{R} dr_{1}
\int_{0}^{\infty} dr_{2} (r_{1} r_{2}) e^{-(r_{1}^{2} + r_{2}^{2})/(4t) + 2r_{1} r_{2}A/(4t)}\notag\\
&=(4 \pi t)^{-1} \int_{0}^{\pi} d \theta_{1} \int_{0}^{2\pi} d \theta_{2} \int_{0}^{R} dr_{1}
\int_{0}^{\infty} dr_{2} (r_{1} r_{2}) e^{-(r_{1}^{2} + r_{2}^{2})/(4t) + 2r_{1} r_{2}A/(4t)}\notag\\
&\ \ \ -(4 \pi t)^{-1} \int_{0}^{\pi} d \theta_{1} \int_{\frac{3\pi}{2}}^{2\pi} d \theta_{2} \int_{0}^{R} dr_{1}
\int_{0}^{\infty} dr_{2} (r_{1} r_{2}) e^{-(r_{1}^{2} + r_{2}^{2})/(4t) + 2r_{1} r_{2}A/(4t)}\notag\\
&=:M_{1} + M_{2}.\label{e61}
\end{align}
Now
\begin{align}
M_{1}&=(4 \pi t)^{-1} \int_{0}^{\pi} d \theta_{1} \int_{0}^{2\pi} d \theta_{2} \int_{0}^{R} dr_{1}
\int_{0}^{\infty} dr_{2} (r_{1} r_{2}) e^{-(r_{1}^{2} + r_{2}^{2})/(4t) + 2r_{1} r_{2}A/(4t)}\notag\\
&= \frac{\pi R^{2}}{2},\label{e62}
\end{align}
and letting  $r_{2} - r_{1} A = \rho$, we have that
\begin{align}
M_{2}&=-(4 \pi t)^{-1} \int_{0}^{\pi} d \theta_{1} \int_{\frac{3\pi}{2}}^{2\pi} d \theta_{2} \int_{0}^{R} dr_{1}
\int_{0}^{\infty} dr_{2} (r_{1} r_{2}) e^{-(r_{1}^{2} + r_{2}^{2})/(4t) + 2r_{1} r_{2}A/(4t)}\notag\\
&=-(4 \pi t)^{-1} \int_{0}^{\pi} d \theta_{1} \int_{-\frac{\pi}{2}}^{0} d \theta_{2} \int_{0}^{R} dr_{1}
\int_{0}^{\infty} dr_{2} (r_{1} r_{2}) e^{-(r_{1}^{2} + r_{2}^{2})/(4t) + 2r_{1} r_{2}A/(4t)}\notag\\
&=-(4 \pi t)^{-1} \int_{0}^{\pi} d \theta_{1} \int_{-\frac{\pi}{2}}^{0} d \theta_{2} \int_{0}^{R} r \, dr
\int_{-Ar}^{\infty} d \rho \, (\rho + Ar) e^{-\rho^{2}/(4t) - r^{2}(1-A^{2})/(4t)}\notag\\
&=-\frac{\pi}{2}t(1-e^{-R^{2}/(4t)})\notag\\
&\ \ \ -(4 \pi t)^{-1} \int_{0}^{\pi} d \theta_{1} \int_{-\frac{\pi}{2}}^{0} d \theta_{2} \int_{0}^{R} dr
\int_{-Ar}^{\infty} d \rho \, Ar^{2} e^{-\rho^{2}/(4t) - r^{2}(1-A^{2})/(4t)}.\label{e63}
\end{align}
We also have that
\begin{align}
&-(4 \pi t)^{-1} \int_{0}^{\pi} d \theta_{1} \int_{-\frac{\pi}{2}}^{0} d \theta_{2} \int_{0}^{R} dr
\int_{-Ar}^{\infty} d \rho \, Ar^{2} e^{-\rho^{2}/(4t) - r^{2}(1-A^{2})/(4t)}\notag\\
&=-(4 \pi t)^{-1} \int_{0}^{\pi} d \theta_{1} \int_{-\frac{\pi}{2}}^{0} d \theta_{2} \int_{0}^{R} dr
\int_{0}^{\infty} d \rho \, Ar^{2} e^{-\rho^{2}/(4t) - r^{2}(1-A^{2})/(4t)}\notag\\
&\ \ \ -(4 \pi t)^{-1} \int_{0}^{\pi} d \theta_{1} \int_{-\frac{\pi}{2}}^{0} d \theta_{2} \int_{0}^{R} dr
\int_{0}^{r} d \rho \, A^{2}r^{2} e^{-A^{2}\rho^{2}/(4t) - r^{2}(1-A^{2})/(4t)}\notag
\end{align}
\begin{align}
&=-\frac{(4 \pi t)^{-1/2}}{2} \int_{0}^{\pi} d \theta_{1} \int_{-\frac{\pi}{2}}^{0} d \theta_{2} \int_{0}^{R} dr
\, Ar^{2} e^{- r^{2}(1-A^{2})/(4t)}\notag\\
&\ \ \ -(4 \pi t)^{-1} \int_{0}^{\pi} d \theta_{1} \int_{-\frac{\pi}{2}}^{0} d \theta_{2} \int_{0}^{R} dr
\int_{0}^{\infty} d \rho \, A^{2}r^{2} e^{-A^{2}\rho^{2}/(4t) - r^{2}(1-A^{2})/(4t)}\notag\\
&\ \ \ +(4 \pi t)^{-1} \int_{0}^{\pi} d \theta_{1} \int_{-\frac{\pi}{2}}^{0} d \theta_{2} \int_{0}^{R} dr
\int_{r}^{\infty} d \rho \, A^{2}r^{2} e^{-A^{2}\rho^{2}/(4t) - r^{2}(1-A^{2})/(4t)}\notag\\
&=-\frac{(4 \pi t)^{-1/2}}{2} \int_{0}^{\pi} d \theta_{1} \int_{-\frac{\pi}{2}}^{0} d \theta_{2} \int_{0}^{R} dr
\, (A + \vert A \vert) r^{2} e^{- r^{2}(1-A^{2})/(4t)}\notag\\
&\ \ \ +(4 \pi t)^{-1} \int_{0}^{\pi} d \theta_{1} \int_{-\frac{\pi}{2}}^{0} d \theta_{2} \int_{0}^{R} dr
\int_{r}^{\infty} d \rho \, A^{2}r^{2} e^{-A^{2}\rho^{2}/(4t) - r^{2}(1-A^{2})/(4t)}\notag\\
&=:N_{1} + N_{2}.\label{e64}
\end{align}
Now $N_1$ equals
\begin{align}
&-\frac{(4 \pi t)^{-1/2}}{2} \int_{0}^{\pi} d \theta_{1} \int_{-\frac{\pi}{2}}^{0} d \theta_{2} \int_{0}^{R} dr
\, (\cos(\theta_{1}-\theta_{2}) + \vert \cos(\theta_{1}-\theta_{2}) \vert) r^{2} e^{- r^{2}(\sin(\theta_{1}-\theta_{2}))^{2}/(4t)}\notag\\
&=-\frac{(4 \pi t)^{-1/2}}{2} \int_{0}^{\pi} d \theta_{1} \int_{0}^{\frac{\pi}{2}} d \theta_{2} \int_{0}^{R} dr
\, (\cos(\theta_{1}+\theta_{2}) + \vert \cos(\theta_{1}+\theta_{2}) \vert) r^{2} e^{- r^{2}(\sin(\theta_{1}+\theta_{2}))^{2}/(4t)}\notag\\
&=-\frac{(4 \pi t)^{-1/2}}{2} \int_{0}^{\pi} d \theta_{1} \int_{\theta_{1}}^{\frac{\pi}{2}+\theta_{1}} d \eta \int_{0}^{R} dr
\, (\cos\eta + \vert \cos\eta \vert) r^{2} e^{- r^{2}(\sin\eta)^{2}/(4t)}\notag\\
&=-(4 \pi t)^{-1/2} \int_{0}^{\frac{\pi}{2}} d \theta_{1} \int_{\theta_{1}}^{\frac{\pi}{2}} d \eta \int_{0}^{R} dr
\, r^{2} \cos\eta \, e^{- r^{2}(\sin\eta)^{2}/(4t)}\notag\\
&\ \ \ -\frac{(4 \pi t)^{-1/2}}{2} \int_{0}^{\frac{\pi}{2}} d \theta_{1} \int_{\frac{\pi}{2}}^{\frac{\pi}{2}+\theta_{1}} d \eta \int_{0}^{R} dr
\, (\cos\eta + \vert \cos\eta \vert) r^{2} e^{- r^{2}(\sin\eta)^{2}/(4t)}\notag\\
&\ \ \ -\frac{(4 \pi t)^{-1/2}}{2} \int_{\frac{\pi}{2}}^{\pi} d \theta_{1} \int_{\theta_{1}}^{\frac{\pi}{2}+\theta_{1}} d \eta \int_{0}^{R} dr
\, (\cos\eta + \vert \cos\eta \vert) r^{2} e^{- r^{2}(\sin\eta)^{2}/(4t)}\notag\\
&=-(4 \pi t)^{-1/2} \int_{0}^{\frac{\pi}{2}} d \theta \int_{\theta}^{\frac{\pi}{2}} d \eta \int_{0}^{R} dr
\, r^{2} \cos\eta \, e^{- r^{2}(\sin\eta)^{2}/(4t)}\notag\\
&=-\frac{R}{\sqrt{\pi}}t^{1/2} + \frac{\pi}{4}t +(4 \pi t)^{-1/2}\int_{0}^{1}d\psi \, (\arcsin \psi - \psi)
\int_{R}^{\infty} dr \, r^{2}e^{- r^{2}\psi^{2}/(4t)} \label{e65}\\
&\ \ \ + O(te^{-R^{2}/(8t)}),\, t \downarrow 0,\notag
\end{align}
by integrating by parts with respect to $\theta$. In addition, as for the computation of $I_{5}$ (see \eqref{e31a}),
we have that
\begin{align*}
N_{2}&=(4 \pi t)^{-1} \int_{0}^{\pi} d \theta_{1} \int_{-\frac{\pi}{2}}^{0} d \theta_{2} \int_{0}^{R} dr
\int_{r}^{\infty} d \rho \, A^{2}r^{2} e^{-A^{2}\rho^{2}/(4t) - r^{2}(1-A^{2})/(4t)}\notag\\
&=\frac{2t}{\pi}\int_{0}^{\pi} d \theta_{1} \int_{-\frac{\pi}{2}}^{0} d \theta_{2} \int_{1}^{\infty} d \rho
\frac{A^{2}}{(A^{2}\rho^{2}+1-A^{2})^{2}} + O(te^{-R^{2}/(4t)})\notag\\
&=\frac{t}{\pi}\int_{0}^{\pi} d \theta_{1} \int_{-\frac{\pi}{2}}^{0} d \theta_{2}
\left(-(\cot(\theta_{1}-\theta_{2}))^{2} + \frac{\cos(\theta_{1}-\theta_{2})}{(\sin(\theta_{1}-\theta_{2}))^{3}}
\arctan (\tan (\theta_{1}-\theta_{2}))\right) + O(te^{-R^{2}/(4t)})\notag\\
&=\frac{t}{\pi}\int_{0}^{\pi} d \theta_{1} \int_{0}^{\frac{\pi}{2}} d \theta_{2}
\left(-(\cot(\theta_{1}+\theta_{2}))^{2} + \frac{\cos(\theta_{1}+\theta_{2})}{(\sin(\theta_{1}+\theta_{2}))^{3}}
\arctan (\tan (\theta_{1}+\theta_{2}))\right) + O(te^{-R^{2}/(4t)})\notag\\
&=\frac{t}{\pi}\int_{0}^{\pi} d \theta_{1} \int_{\theta_{1}}^{\frac{\pi}{2}+\theta_{1}} d \eta
\left(-(\cot\eta)^{2} + \frac{\cos\eta}{(\sin\eta)^{3}}\arctan (\tan \eta)\right) + O(te^{-R^{2}/(4t)})\notag\\
&=\frac{t}{\pi}\int_{0}^{\pi} d \theta \int_{\theta}^{\frac{\pi}{2}+\theta} d \eta
\left(-(\cot\eta)^{2} + \frac{\cos\eta}{(\sin\eta)^{3}}\arctan (\tan \eta)\right) + O(te^{-R^{2}/(4t)})
,\, t \downarrow 0.
\end{align*}
Note that
\begin{equation*}
\arctan(\tan\eta) =
\begin{cases}
\eta & \text{if $\eta \in (0,\frac{\pi}{2})$};\\
\eta - \pi & \text{if $\eta \in (\frac{\pi}{2},\frac{3\pi}{2})$}.
\end{cases}
\end{equation*}
Hence
\begin{align}
N_{2}
&=\frac{t}{\pi}\int_{0}^{\pi} d \theta \int_{\theta}^{\frac{\pi}{2}+\theta} d \eta
\left(-(\cot\eta)^{2} + \frac{\cos\eta}{(\sin\eta)^{3}}\arctan (\tan \eta)\right) + O(te^{-R^{2}/(4t)})\notag\\
&=\frac{t}{\pi}\int_{0}^{\frac{\pi}{2}} d \theta \int_{\theta}^{\frac{\pi}{2}} d \eta
\left(-(\cot\eta)^{2} + \frac{\eta\cos\eta}{(\sin\eta)^{3}}\right)\notag\\
&\ \ \ +\frac{t}{\pi}\int_{0}^{\frac{\pi}{2}} d \theta \int_{\frac{\pi}{2}}^{\frac{\pi}{2}+\theta} d \eta
\left(-(\cot\eta)^{2} + \frac{(\eta - \pi)\cos\eta}{(\sin\eta)^{3}}\right)\notag\\
&\ \ \ +\frac{t}{\pi}\int_{\frac{\pi}{2}}^{\pi} d \theta \int_{\theta}^{\frac{\pi}{2}+\theta} d \eta
\left(-(\cot\eta)^{2} + \frac{(\eta - \pi)\cos\eta}{(\sin\eta)^{3}}\right) + O(te^{-R^{2}/(4t)})\notag\\
&=\frac{t}{2\pi} + \frac{t}{2\pi} + \frac{t}{\pi}\left(\frac{\pi^{2}}{4}-1\right)+ O(te^{-R^{2}/(4t)})\notag\\
&=\frac{\pi}{4}t + O(te^{-R^{2}/(4t)}),\, t \downarrow 0.\label{e68}
\end{align}
Thus, by \eqref{e63}, \eqref{e64}, \eqref{e65} and \eqref{e68}, we have that, as $t \downarrow 0$,
\begin{equation}\label{e69}
M_{2}
=-\frac{R}{\sqrt{\pi}}t^{1/2} +(4 \pi t)^{-1/2}\int_{0}^{1}d\psi \, (\arcsin \psi - \psi)
\int_{R}^{\infty} dr \, r^{2}e^{- r^{2}\psi^{2}/(4t)}+O(te^{-R^{2}/(8t)}).
\end{equation}
Therefore, by \eqref{e61}, \eqref{e62} and \eqref{e69}, as $t \downarrow 0$,
\begin{align*}
&(4 \pi t)^{-1} \int_{0}^{\pi} d \theta_{1} \int_{0}^{\frac{3\pi}{2}} d \theta_{2} \int_{0}^{R} dr_{1}
\int_{0}^{\infty} dr_{2} (r_{1} r_{2}) e^{-(r_{1}\cos \theta_{1} - r_{2} \cos \theta_{2})^{2}/(4t)
-(r_{1}\sin \theta_{1} - r_{2} \sin \theta_{2})^{2}/(4t)}
\\
&=\frac{\pi R^{2}}{2} -\frac{R}{\sqrt{\pi}}t^{1/2} +(4 \pi t)^{-1/2}\int_{0}^{1}d\psi \, (\arcsin \psi - \psi)
\int_{R}^{\infty} dr \, r^{2}e^{- r^{2}\psi^{2}/(4t)}+O(te^{-R^{2}/(8t)}).\notag
\end{align*}
\end{proof}
We note that
\begin{align*}
\mathcal{V}_{\frac{3\pi}{2}}(t;R) &= \int_{B_{\frac{3\pi}{2}}(R)} dx \int_{W_{\frac{3\pi}{2}}} dy \, p(x,y;t)\notag\\
&=\int_{B_{\pi}(R)} dx \int_{W_{\frac{3\pi}{2}}} dy \, p(x,y;t) + \int_{B_{\frac{\pi}{2}}(R)} dx \int_{W_{\frac{3\pi}{2}}} dy \, p(x,y;t).
\end{align*}
By Theorem \ref{T1} and Lemma \ref{L6.1}, this implies that, as $t \downarrow 0$,
\begin{align}
&\int_{B_{\frac{\pi}{2}}(R)} dx \int_{W_{\frac{3\pi}{2}}} dy \, p(x,y;t)\notag\\
&= \frac{\pi R^{2}}{4} - \frac{R}{\sqrt{\pi}} t^{1/2}+g\left(\frac{3\pi}{2}\right)t \notag\\
&\ \ \ +(4 \pi t)^{-1/2}\int_{0}^{1}d\psi \, (\arcsin \psi - \psi)
\int_{R}^{\infty} dr \, r^{2}e^{- r^{2}\psi^{2}/(4t)}+O(te^{-R^{2}/(8t)})\notag\\
&= \frac{\pi R^{2}}{4} - \frac{R}{\sqrt{\pi}} t^{1/2}+\frac{t}{\pi}\label{e69c}\\
&\ \ \ +(4 \pi t)^{-1/2}\int_{0}^{1}d\psi \, (\arcsin \psi - \psi)
\int_{R}^{\infty} dr \, r^{2}e^{- r^{2}\psi^{2}/(4t)}+O(te^{-R^{2}/(8t)}).\notag
\end{align}
The result of Lemma~\ref{L6.1} and formula \eqref{e69c} will be used in Section \ref{S7} below.

\section{The heat content of the fractal polyhedron $D_{s}$.}\label{S7}
In this section, we use Theorem~\ref{T1} to compute the heat content of the fractal polyhedron $D_{s}$
which was constructed in Section~\ref{S1}. To do this, we adapt the scheme of \cite{vdBdH99} to the three-dimensional
setting below. The key step in \cite{vdBdH99} was to obtain a renewal equation by making a
suitable Ansatz for the heat content. The corresponding Ansatz has been made here in \eqref{e84} and \eqref{e89}
for $0<s<\sqrt{2}-1, s \neq \frac{1}{5}$, and in \eqref{e94} and \eqref{e98} for $s=\frac{1}{5}$.

In order to derive the required renewal equation, we need to compute the contribution to the heat content
$H_{D_{s}}(t)$ from $Q_{0}$ and $Q_{1,1}$. We do this below in Lemma~\ref{L7.1} and Lemma~\ref{L7.2} respectively.
In what follows, for $A \subset \R^{3}$, $d(x,A)$ is the 3-dimensional analogue of \eqref{e14a}. We make the
following approximations for $u_{D_{s}}(x;t)$.

Let $0 < \delta \leq  \min\left\{\frac{s^{2}}{2}, \frac{s(1-s)}{4}\right\}$ and let $x \in D_{s}$.
If $d(x,\partial D_{s}) \geq \delta$, then we have that
\begin{equation*}
\vert u_{D_{s}}(x;t) - 1 \vert \leq 2^{3/2}e^{-\delta^{2}/(8t)},
\end{equation*}
by the principle of not feeling the boundary, \cite[Proposition 9(i)]{mvdB13}. We define
\begin{equation*}
\tilde{F}=\{x \in D_{s} : d(x,\partial D_{s}) <\delta, d(x,e) > \delta \text{ for all edges } e\in \partial D_{s}\}.
\end{equation*}
If $x \in \tilde{F}$, then we have that
\begin{equation}\label{e69f}
\vert u_{D_{s}}(x;t) - u_{H}(x;t) \vert \leq 2^{3/2}e^{-\delta^{2}/(8t)},
\end{equation}
where
\begin{equation*}
u_{H}(x;t)=(4\pi t)^{-1/2}\int_{-d(x,\partial D_{s})}^{\infty} d\zeta \, e^{-\zeta^{2}/(4t)},
\end{equation*}
i.e. $H$ is a half-space whose boundary contains the face of $\partial D_{s}$ nearest to $x$.
Let
\begin{equation*}
\tilde{E}=\{x \in D_{s} : d(x,e) <\delta \text{ for some edge } e\in \partial D_{s}, d(x,v) >
\delta \text{ for all vertices } v\in \partial D_{s}\}.
\end{equation*}
If $x \in \tilde{E}$, then we have that
\begin{equation}\label{e69i}
\vert u_{D_{s}}(x;t) - u_{W}(x;t) \vert \leq 2^{3/2}e^{-\delta^{2}/(8t)},
\end{equation}
where $W$ is the infinite wedge $W_{\frac{\pi}{2}}$ for entrant edges and $W$ is the infinite wedge
$W_{\frac{3\pi}{2}}$ for re-entrant edges. (See the proof of Lemma~\ref{L7.1} for further details).
The estimates \eqref{e69f}, \eqref{e69i} follow by similar arguments to those given in the proof of Lemma~\ref{L}
with $\tilde{D}=D_{s}$, $F=H, W$, $E= \tilde{F}, \tilde{E}$ respectively and $G=\{x \in D_{s} : d(x,E)<\delta\}$.

It remains to approximate $u_{D_{s}}(x;t)$ for $x$ near a vertex of $\partial Q_{0} \cap \partial D_{s}$,
$\partial Q_{1,1} \cap \partial D_{s}$ respectively. We only require the contribution to the heat content
$H_{D_{s}}(t)$ from these vertices to derive the required renewal equation.
The relevant approximation to make here is via a one-sided infinite cone $C_{v}$ with vertex $v \in \partial D_{s}$
such that $\partial C_{v} \supseteq \{x \in \partial D_{s} : d(x,v) < \delta\}$.
Definition aside, no viable expressions are known for $u_{C_{v}}(x;t)$ in this 3-dimensional setting.
For our purposes, it is sufficient to approximate the neighbourhood of each vertex $v \in
(\partial Q_{0} \cup \partial Q_{1,1})\cap\partial D_{s}$ by a cube $S_{v}$. Each cube $S_{v}$ centred at $v$
has side-length $2\delta$ and is chosen such that the faces of $\partial S_{v}$ are pairwise-parallel to those of
$\partial Q_{0}$.
We are interested in the contribution to the heat content $H_{D_{s}}(t)$ from the region $S_{v} \cap D_{s}$.
There are two cases to consider. Either the vertex $v$ is entrant and $S_{v} \cap D_{s}$ is $\frac{1}{8}$ of $S_{v}$,
or the vertex $v$ is re-entrant and $S_{v} \cap D_{s}$ is $\frac{5}{8}$ of $S_{v}$. If $v$ is entrant, then the coefficient
of $t^{3/2}$ in the expansion for $H_{D_{s}}(t)$ is equal to $\frac{-1}{\pi^{3/2}}$ by separation of variables. Unfortunately,
we were unable to compute the coefficient of $t^{3/2}$ for a re-entrant vertex. However, the contribution to the heat content
$H_{D_{s}}(t)$ from each region $S_{v} \cap D_{s}$ is of order: $\delta^3$ from the volume, $\delta^2t^{1/2}$ from the
surface area of the adjoining faces of the vertex, and $\delta t$ from the length of the adjoining edges of the vertex.
Thus, if we choose $\delta$ as follows:
\begin{equation}\label{e70a}
\delta = 8t^{1/2}(\log(t^{-1}))^{1/2},
\end{equation}
then the contribution to the heat content $H_{D_{s}}(t)$ from each region $S_{v} \cap D_{s}$ is
\newline
$O(t^{3/2}(\log(t^{-1}))^{3/2})$. We note that this choice of $\delta$ gives
$O(e^{-\delta^{2}/(8t)})=O(t^{8})$.

\begin{lemma}\label{L7.1}
Let $0<s<\sqrt{2}-1$. Then
\begin{equation*}
\int_{Q_{0}} dx \, u_{D_{s}}(x;t) = 1 -6(1-s^{2})\frac{t^{1/2}}{\sqrt{\pi}} + \frac{12}{\pi}t
+O(t^{3/2}(\log(t^{-1}))^{3/2},\, t \downarrow 0.
\end{equation*}
\end{lemma}

\begin{proof}
Partition $Q_{0}$ into the following sets.
\newline
(i) $\partial Q_{0} \cap \partial D_{s}$ has 32 vertices; $v_{i}, i=1,\dots,32$. At each vertex $v_{i}$,
consider a cube $S_{i}$ of side-length $2\delta$ centred at $v_{i}$. Let
$\tilde{S}_{i}=S_{i} \cap Q_{0}, i=1,\dots,32$ and $\tilde{S}=\cup_{i=1}^{32}\tilde{S}_{i}$.
\newline
(ii) $\partial Q_{0} \cap \partial D_{s}$ has 36 edges; $e_{j}, j=1,\dots,36$. Let $\tilde{E}_{j}=\{x \in Q_{0} :
d(x,e_{j})<\delta, x \not\in \tilde{S}\}$.
\newline
(iii) $\tilde{F}=\left\{x \in Q_{0} : d(x,\partial Q_{0} \cap \partial D_{s})<\delta, x \not\in \left(\tilde{S} \cup \bigcup_{j=1}^{36} \tilde{E}_{j}\right)\right\}$.
\newline
(iv) The interior of $Q_{0}$ minus (i), (ii) and (iii); an open polygon $P_{\delta}$ with distance at least $\delta$ to
$\partial D_{s}$.
\newline
(v) The remainder, which has measure zero.
\newline
The contribution to the heat content from (iv) is $\vert P_{\delta}\vert + O(e^{-\delta^{2}/(8t)}) = \vert P_{\delta}\vert
+ O(t^{3/2}),\, t \downarrow 0$.

To compute the contribution from (ii), there are two types of edges to consider. Each edge $e_{j}$ is the intersection
of two faces of $\partial D_{s}$. For fixed $j$, let $\Pi_{j}$ denote the plane which is orthogonal to these faces and
intersects $e_{j}$ in exactly one point. Apply Lemma~\ref{L} with $\tilde{D}=D_{s} \cap \Pi_{j}$, $F=W_{\frac{\pi}{2}}, W_{\frac{3\pi}{2}}$ respectively,
$E=\tilde{E}_{j} \cap \Pi_{j}$ and $G=\{x \in D_{s} \cap \Pi_{j} : d(x, \tilde{E}_{j} \cap \Pi_{j})<\delta\}$
so that either;
\newline
(I) the contribution from $\tilde{E}_{j} \cap \Pi_{j}$ can be approximated by that from a sector $B_{\delta}(\frac{\pi}{2})$ in a wedge
$W_{\frac{\pi}{2}}$, or
\newline
(II) the contribution from $\tilde{E}_{j} \cap \Pi_{j}$ can be approximated by that from a sector $B_{\delta}(\pi)$ in a wedge $W_{\frac{3\pi}{2}}$.

We can use Lemma~\ref{L3.1} to deduce that the contribution from the edges of type (I) is
\begin{align*}
&12(1-2\delta)\left(\vert \tilde{E}_{j} \cap \Pi_{j} \vert - 2\delta\frac{t^{1/2}}{\sqrt{\pi}}+\frac{1}{\pi}t\right)+O(t^{3/2})\notag\\
&=12\left(\vert \tilde{E}_{j} \vert -2\delta(1-2\delta)\frac{t^{1/2}}{\sqrt{\pi}} +\frac{1}{\pi}(1-2\delta)t\right)+O(t^{3/2}),\,
t \downarrow 0.
\end{align*}
In addition, Lemma~\ref{L6.1} gives that the contribution from the edges of type (II) is
\begin{align*}
&24(s-2\delta)\left(\vert \tilde{E}_{j} \cap \Pi_{j} \vert - \delta\frac{t^{1/2}}{\sqrt{\pi}}\right)+O(t^{3/2})\notag\\
&=24\left(\vert \tilde{E}_{j} \vert - \delta(s-2\delta)\frac{t^{1/2}}{\sqrt{\pi}}\right)+O(t^{3/2}),\, t \downarrow 0.
\end{align*}

Each cross-section of $\tilde{F}$ is a union of rectangles and cuspidal regions. Thus by Section~\ref{S4},
we have that the contribution from $\tilde{F}$ equals
\begin{equation*}
\vert \tilde{F} \vert - \mathcal{H}^2(\partial \tilde{F}) \frac{t^{1/2}}{\sqrt{\pi}} +O(t^{3/2}),\, t \downarrow 0.
\end{equation*}
For edges of type (I), the sector $\tilde{E}_{j}\cap\Pi_{j}$ has two cuspidal neighbours. For edges of type (II), the sector
$\tilde{E}_{j}\cap\Pi_{j}$ has one cuspidal neighbour but by Lemma~\ref{L6.1}, the term of order $t^{3/2}$ is half that
of Lemma~\ref{L}. Even though the sector and cusp terms of order $t^{3/2}$ cancel out, the remainder is dominated by the
contribution from (i), which is
$O(t^{3/2}(\log (t^{-1}))^{3/2}).$ This completes the proof of Lemma~\ref{L7.1}.
\end{proof}

Following the strategy of \cite{vdBdH99}, in order to compute $\int_{D_{s}- Q_{0}} dx \, u_{D_{s}}(x;t)$,
we introduce a model solution which approximates $u_{D_{s}}$ in one of the six components of $D_{s}- Q_{0}$.
Consider the half-space $H=\{(x_{1},x_{2},x_{3})\in\R^{3} : x_{1}<0\}$ and attach one of the six components
of $D_{s}- Q_{0}$ to $H$. The resulting set is
\begin{equation*}
H_{s} = \text{interior} \left\{\overline{H \cup \left[\bigcup_{j\geq 1} \bigcup_{1\leq i \leq \frac{1}{6}N(j)}
Q_{j,i}\right]}\right\}.
\end{equation*}
Let $u_{H_{s}}$ be the unique solution of \eqref{e3} and \eqref{e4} with $D=H_{s}$. Define
\begin{equation*}
E(t) = \int_{H_{s}- H} dx \, u_{H_{s}}(x;t).
\end{equation*}
Applying Lemma~\ref{L} with $\tilde{D}=D_{s}$, $F=H_{s}$, $E=H_{s}-H$ and $G=\{x \in H_{s} : d(x,E)<\epsilon\}$,
where $\epsilon=\frac{(1+s^{2})^{1/2}}{2(1-s)}(1-2s-s^{2})$, we have
\begin{align}
\int_{D_{s}-Q_{0}} dx \, u_{D_{s}}(x;t) &= 6\int_{H_{s}-H} dx \, u_{D_{s}}(x;t) \notag\\
&= 6\int_{H_{s}-H} dx \, u_{H_{s}}(x;t) +O(e^{-\epsilon^{2}/(8t)})\notag\\
&=6E(t) + O(e^{-\epsilon^{2}/(8t)}),\, t \downarrow 0. \label{e76}
\end{align}

In contrast to \cite{vdBdH99}, where the temperature of the boundary is fixed for all $t > 0$, we must
account for the fact that $H_{s}-H$ and $H_{s}-D_{s}$ feel each other's presence. The choice of $\epsilon$
above is a lower bound for the distance between $H_{s}-H$ and $H_{s}-D_{s}$.

Similarly to Lemma~\ref{L7.1}, we have;
\begin{lemma}\label{L7.2}
Let $0<s<\sqrt{2}-1$. Then
\begin{equation*}
\int_{Q_{1,1}}dx \, u_{H_{s}}(x;t) = s^{3} - 5s^{2}(1-s^{2})\frac{t^{1/2}}{\sqrt{\pi}} + \frac{12s}{\pi}t +
\tilde{h}(t),\, t \downarrow 0,
\end{equation*}
where $\vert \tilde{h}(t) \vert \leq \tilde{C} t^{3/2}(\log(t^{-1}))^{3/2}$ for $t\leq s^{2}$ and some constant $\tilde{C}>0$.
\end{lemma}

\begin{proof}
The proof of Lemma~\ref{L7.2} is analogous to that of Lemma~\ref{L7.1}.
We note that in this case there is an additional type of edge to consider.
Namely, the edges where the contribution from ($\tilde{E} \cap Q_{1,1})\cap\Pi_{j}$ can be approximated by
that from a sector $B_{\frac{\pi}{2}}(\delta)$ in a wedge $W_{\frac{3\pi}{2}}$. There are 4 such edges so the
contribution to the heat content is
\begin{align*}
&4(s-2\delta)\left(\vert (\tilde{E} \cap Q_{1,1}) \cap\Pi_{j} \vert - \delta\frac{t^{1/2}}{\sqrt{\pi}} + \frac{1}{\pi}t\right) +O(t^{3/2}) \notag\\
&=4\left(\vert \tilde{E} \cap Q_{1,1} \vert - \delta(s-2\delta)\frac{t^{1/2}}{\sqrt{\pi}} + \frac{(s-2\delta)}{\pi}t\right) +O(t^{3/2}),\,
t \downarrow 0, 
\end{align*}
by \eqref{e69c}. This completes the proof of Lemma \ref{L7.2} by our choice of $\delta$, \eqref{e70a}.
\end{proof}

Below we state and prove the corresponding 3-dimensional result to \cite[Proposition 4]{vdBdH99}
for completeness.

\begin{lemma}\label{L7.3}
Fix $0<s<\sqrt{2}-1$. Then
\begin{equation}\label{e79}
E(t)=5s^{3}E\left(\frac{t}{s^{2}}\right)+s^{3}-5s^{2}(1-s^{2})\frac{t^{1/2}}{\sqrt{\pi}} +\frac{12s}{\pi}t+h(t),
\end{equation}
where $\vert h(t) \vert \leq \tilde{C} t^{3/2}(\log(t^{-1}))^{3/2}$ for $t\leq s^{2}$ and some constant $\tilde{C}>0$.
\end{lemma}

\begin{proof}
Re-write $E(t)$ as
\begin{equation*}
E(t)=\int_{Q_{1,1}} dx \, u_{H_{s}}(x;t) + \int_{H_{s}-(H \cup Q_{1,1})} dx \, u_{H_{s}}(x;t).
\end{equation*}
Then $H_{s}-(H \cup Q_{1,1})$ consists of 5 copies of $H_{s} - H$ scaled by a factor $s$, say $A_{1}, \dots, A_{5}$.
Each of these copies has a face $f_{i}$ connecting it to $Q_{1,1}$. Let $H_{A_{i}}$ be the half-space such that
$\partial H_{A_{i}} \supset f_{i}$ and $H_{A_{i}} \supset Q_{1,1}$. Put $F_{i} = A_{i} \cup H_{A_{i}} \cup f_{i}$.
Then $F_{i}$ is a copy of $H_{s}$ scaled by a factor $s$. Thus, by scaling, we have
\begin{equation*}
\int_{A_{i}}dx \, u_{F_{i}}(x;t) = s^{3}\int_{H_{s}-H}dx \, u_{H_{s}}(x;t/s^{2}) = s^{3}E(t/s^{2}).
\end{equation*}
Define $G_{i}=\{x \in H_{s} : d(x,A_{i})<\tilde{\epsilon}\}$, where $\tilde{\epsilon}=\frac{s(1-2s-s^{2})}{2(1-s)}$.
Applying Lemma~\ref{L} with $\tilde{D}=H_{s}$, $F=F_{i}$, $E=A_{i}$ and $G=G_{i}$, we have
\begin{equation*}
\int_{A_{i}}dx \, u_{H_{s}}(x;t)=\int_{A_{i}}dx \, u_{F_{i}}(x;t) + O(e^{-\tilde{\epsilon}^{2}/(8t)}),\, t \downarrow 0.
\end{equation*}
Hence
\begin{align*}
\int_{H_{s}-(H \cup Q_{1,1})} dx \, u_{H_{s}}(x;t) &= 5\int_{A_{i}} dx \, u_{H_{s}}(x;t)\notag\\
&=5\int_{A_{i}}dx \, u_{F_{i}}(x;t) + O(e^{-\tilde{\epsilon}^{2}/(8t)})\notag\\
&=5s^{3}E(t/s^{2}) + O(e^{-\tilde{\epsilon}^{2}/(8t)}),\, t \downarrow 0.
\end{align*}
Combining this with Lemma~\ref{L7.2} gives the result.
\end{proof}

We must now consider the different regimes for $s$.

\begin{lemma}\label{L7.4}
Let $d=\frac{3}{2} + \frac{1}{2}\frac{\log 5}{\log s}$ and fix $0<s<\sqrt{2}-1, s \neq \frac{1}{5}$.
Then there exists a periodic, continuous function $p_{s}:\R \rightarrow \R$ with period $\log(s^{-2})$ such that
as $t \downarrow 0$,
\begin{equation}\label{e83}
E(t)= \frac{s^{3}}{1-5s^{3}} - \frac{5s^{2}(1-s^{2})}{1-5s^{2}}\frac{t^{1/2}}{\sqrt{\pi}}+\frac{12s}{\pi(1-5s)}t
+p_{s}(\log t)t^{d} + O(t^{3/2}(\log (t^{-1}))^{3/2}).
\end{equation}
\end{lemma}

\begin{proof}
Define
\begin{equation}\label{e84}
q_{s}(t)t^{d} = E(t) -\frac{s^{3}}{1-5s^{3}} + \frac{5s^{2}(1-s^{2})}{1-5s^{2}}\frac{t^{1/2}}{\sqrt{\pi}}
- \frac{12s}{\pi(1-5s)}t.
\end{equation}
Substitute \eqref{e79} into \eqref{e84} to obtain
\begin{equation*}
q_{s}(t)t^{d}=5s^{3}E\left(\frac{t}{s^{2}}\right) - 5s^{3}\left(\frac{s^{3}}{1-5s^{3}}\right)
+25s^{4}\left(\frac{1-s^{2}}{1-5s^{2}}\right)\frac{t^{1/2}}{\sqrt{\pi}} - 5s\left(\frac{12s}{\pi(1-5s)}\right)t
+h(t).
\end{equation*}
From \eqref{e84}, we also have
\begin{equation*}
q_{s}\left(\frac{t}{s^{2}}\right)\frac{t^{d}}{5s^{3}} = E\left(\frac{t}{s^{2}}\right) - \frac{s^{3}}{1-5s^{3}}
+\frac{5s(1-s^{2})}{1-5s^{2}}\frac{t^{1/2}}{\sqrt{\pi}} - \frac{12}{\pi s (1-5s)}t,
\end{equation*}
which implies
\begin{equation*}
q_{s}\left(\frac{t}{s^{2}}\right)t^{d}=q_{s}(t)t^{d} - h(t),
\end{equation*}
or equivalently,
\begin{equation*}
q_{s}(t)=q_{s}\left(\frac{t}{s^{2}}\right) + h(t)t^{-d}.
\end{equation*}
Similarly to \cite[Proposition 5]{vdBdH99}, define $p_{s}(\log t)$ by
\begin{equation}\label{e89}
q_{s}(t)=p_{s}(\log t) - \sum_{j=1}^{\infty} h(ts^{2j})(ts^{2j})^{-d}.
\end{equation}
Then
\begin{equation*}
q_{s}\left(\frac{t}{s^{2}}\right) = p_{s}\left(\log \frac{t}{s^{2}}\right) - h(t)t^{-d}
- \sum_{j=1}^{\infty} h(ts^{2j})(ts^{2j})^{-d},
\end{equation*}
hence
\begin{equation*}
p_{s}\left(\log \frac{t}{s^{2}}\right) = q_{s}(t) + \sum_{j=1}^{\infty} h(ts^{2j})(ts^{2j})^{-d}
=p_{s}(\log t).
\end{equation*}
Since there is a constant $\tilde{C}>0$ such that $\vert h(t) \vert \leq \tilde{C} t^{3/2}(\log(t^{-1}))^{3/2}$
for $t\leq s^{2}$, there is a constant $\hat{C}>0$ such that
\begin{align}
\bigg\vert \sum_{j=1}^{\infty} h(ts^{2j})(ts^{2j})^{-d} \bigg\vert
&\leq \hat{C}t^{(3/2)-d}\left((\log(t^{-1}))^{3/2}\sum_{j=1}^{\infty}(s^{3-2d})^{j} + (-\log s)^{3/2}\sum_{j=1}^{\infty}(2j)^{3/2}(s^{3-2d})^{j}\right)\notag\\
&=O(t^{(3/2)-d}(\log(t^{-1}))^{3/2}).\label{e92}
\end{align}
Combining \eqref{e84} with \eqref{e89} and \eqref{e92} gives \eqref{e83}.
\end{proof}

\begin{lemma}\label{L7.5}
Let $s=\frac{1}{5}$. Then there exists a periodic, continuous function
$p_{\frac{1}{5}}:\R \rightarrow \R$ with period $\log 25$ such that as $t \downarrow 0$,
\begin{equation*}
E(t) = \frac{1}{120} - \frac{6}{25}\frac{t^{1/2}}{\sqrt{\pi}} - \frac{6}{5\pi\log 5} t \log t + \frac{12}{5\pi}t
+p_{\frac{1}{5}}(\log t)t + O(t^{3/2}(\log(t^{-1}))^{3/2}).
\end{equation*}
\end{lemma}

\begin{proof}
For $s=\frac{1}{5}$, $d=1$. Define
\begin{equation}\label{e94}
q_{\frac{1}{5}}(t)t = E(t) -\frac{1}{120} + \frac{6}{25}\frac{t^{1/2}}{\sqrt{\pi}} + \frac{6}{5\pi\log 5} t \log t - \frac{12}{5\pi}t.
\end{equation}
Substitute \eqref{e79} with $s=\frac{1}{5}$ into \eqref{e94} to obtain
\begin{equation*}
q_{\frac{1}{5}}(t)t = \frac{1}{25}E(25t)- \frac{1}{3000} + \frac{6}{125}\frac{t^{1/2}}{\sqrt{\pi}} +
\frac{6}{5\pi\log 5}t \log t + h(t).
\end{equation*}
By considering \eqref{e94} with $t$ replaced by $25t$, we obtain that
\begin{equation*}
q_{\frac{1}{5}}(25t)t = q_{\frac{1}{5}}(t)t - h(t),
\end{equation*}
or equivalently,
\begin{equation}\label{e97}
q_{\frac{1}{5}}(t) = q_{\frac{1}{5}}(25t) + h(t)t^{-1}.
\end{equation}
Define $p_{\frac{1}{5}}(\log t)$ by
\begin{equation}\label{e98}
q_{\frac{1}{5}}(t) = p_{\frac{1}{5}}(\log t) - \sum_{j=1}^{\infty} h(t(25)^{-j})(t(25)^{-j})^{-1}.
\end{equation}
Then, as in the proof of Lemma~\ref{L7.4}, we have $p_{\frac{1}{5}}(\log t) = p_{\frac{1}{5}}(\log 25t)$
via \eqref{e97}. Since \eqref{e92} holds for $s=\frac{1}{5}$, we obtain the same remainder estimate as
in Lemma~\ref{L7.4}.
\end{proof}

Combining Lemma~\ref{L7.1}, \eqref{e76} and Lemma~\ref{L7.4}, Lemma~\ref{L7.5} we obtain \eqref{e13f}, \eqref{e13g}
respectively. Using the fact that $t\mapsto H_{D_{s}}(t)$ is continuous, it can be shown that $t \mapsto p_{s}(\log t)$
is continuous for $0<s<\sqrt{2}-1$, see \cite[Section 5.1]{vdBdH99}.


\begin{thebibliography}{99}

\bibitem{mvdB2006}
van den Berg, M.: Heat content and Hardy inequality
for complete Riemannian manifolds. Journal of Functional Analysis
\textbf{233}, 478--493 (2006).

\bibitem{mvdB2007}
van den Berg, M.: Heat flow and Hardy inequality in
complete Riemannian manifolds with singular initial conditions.
Journal of Functional Analysis \textbf{250}, 114--131 (2007).

\bibitem{mvdB13}
van den Berg, M.: Heat Flow and Perimeter in $\mathbb{R}^{m}$.
Potential Anal. \textbf{39}, 369--387 (2013).

\bibitem{vdBG2} van den Berg, M., Gilkey, P.:
Heat flow out of a compact manifold. Journal of Geometric
Analysis DOI 10.1007/s12220-014-9485-2.


\bibitem{vdBkG14}
van den Berg, M., Gittins, K.:
Uniform bounds for the heat content of open sets in Euclidean space. Differ. Geom. Appl.
\textbf{40}, 67--85 (2015).

\bibitem{vdBdH99}
van den Berg, M., den Hollander, F.:
Asymptotics for the heat content of a planar region with a fractal polygonal boundary.
Proc. London Math. Soc. \textbf{78}, 627--661 (1999).

\bibitem{vdBS88}
van den Berg, M., Srisatkunarajah, S.: Heat equation for a region in $\R^2$ with a polygonal boundary.
J. London Math. Soc. \textbf{37}, 119--127 (1988).

\bibitem{vdBS90}
van den Berg, M., Srisatkunarajah, S.:
Heat flow and Brownian motion for a region in $\R^{2}$ with a polygonal boundary.
Probab. Th. Rel. Fields \textbf{86}, 41--52 (1990).

\bibitem{Ca45}
Carslaw, H. S., Jaeger, J. C.:
Conduction of Heat in Solids. Clarendon Press, Oxford
(2000).

\bibitem{LCE}
Evans, L. C.: Partial Differential Equations.
Graduate Studies in Mathematics, \textbf{19}, American Mathematical
Society, Providence RI (2002).

\bibitem{EG92}
Evans, L. C., Gariepy, R.F.:
Measure Theory and Fine Properties of Functions.
Chapman and Hall/ CRC, Boca Raton (1992).

\bibitem{G}
Gilkey, P.B.:
Asymptotic Formulae in Spectral Geometry. Chapman \& Hall / CRC, Boca Raton (2004).

\bibitem{GR}
Gradshteyn, I. S., Ryzhik, I. M.:
Table of Integrals, Series and Products. Seventh Edition,
Academic Press (2007).

\bibitem{K}
Kac, M.: Can one hear the shape of a drum? Amer. Math. Monthly \textbf{73}, 1--23 (1966).

\bibitem{MPPP1}
M. Miranda Jr., D. Pallara, F. Paronetto, M. Preunkert,
On a characterisation of perimeters in $\R^N$ via heat
semigroup. Ricerche di Matematica \textbf{54}, 615--621 (2005).

\bibitem{MPPP07}
Miranda, M. Jr., Pallara, D., Paronetto, F., Preunkert, M.:
Short-time heat flow and functions of bounded variation in
$\mathbb{R}^{N}$. Annales de la Facult\'{e} des Sciences de
Toulouse \textbf{16}, 125--145 (2007)

\bibitem{P}
Preunkert, M.: A semigroup version of
the isoperimetric inequality. Semigroup Forum \textbf{68},
233--245 (2004).

\end{thebibliography}
\end{document}